\documentclass[onecolumn]{IEEEtran}

\usepackage[cmex10]{amsmath}
\usepackage{latexsym}
\newtheorem{theorem}{Theorem}
\newtheorem{proposition}{Proposition}

\newtheorem{lemma}{Lemma}

\newtheorem{definition}{Definition}

\newtheorem{claim}{Claim}

\newtheorem{remark}{Remark}

\newcommand{\mathbb}[1]{{\bf #1}}

\newcommand{\thmref}[1]{Theorem~\ref{thm:#1}} 
\newcommand{\lemref}[1]{Lemma~\ref{lem:#1}} 
\newcommand{\propref}[1]{Proposition~\ref{prop:#1}} 
\newcommand{\remref}[1]{Remark~\ref{rem:#1}} 
\newcommand{\defref}[1]{Definition~\ref{def:#1}} 
\newcommand{\secref}[1]{Section~\ref{sec:#1}} 
\newcommand{\eqnref}[1]{(\ref{eq:#1})} 

\def\be{\begin{equation} }
\def\ee{ \end{equation}}

\def\ben{\begin{equation*}}
\def\een{\end{equation*}}
\def\bea{\begin{eqnarray}}
\def\eea{\end{eqnarray}}
\def\ee{\end{eqnarray}}
\def\bean{\begin{eqnarray*}}
\def\eean{\end{eqnarray*}}

\newcommand\ignore[1]{}

\def\R{\mathbb{R}} 
\def\Z{\mathbb{Z}} 
\def\N{\mathbb{N}} 

\newcommand{\Ex}[1]{\mathbb{E}\left[#1\right]} 

\newcommand{\Prp}[2]{\mathbb{P}_{#1}\left(#2\right)} 

\newcommand{\Prpwo}[1]{\mathbb{P}_{#1}} 

\newcommand{\Ind}[1]{{\rm {\bf 1}}_{\{#1\}}} 
\renewcommand{\Pr}[1]{\mathbb{P}\left(#1\right)} 

\newcommand{\bigoh}[1]{O\left(#1\right)}

\def\sF{\mathcal{F}}

\def\sR{\mathcal{R}}
\def\sS{\mathcal{S}}\def\sT{\mathcal{T}}

\newcommand\QED{\ifhmode\allowbreak\else\nobreak\fi
\quad\nobreak$\Box$\medbreak}
\newcommand{\proofstart}{\par\noindent\sl Proof:\rm\enspace}
\newcommand{\proofend}{\QED\par}

\def\eps{\epsilon}

\def\bX{{\bf X}}
\def\conf{\widehat{{\sf cf}}}
\def\barconf{\overline{{\sf cf}}}

\def\hatp{\widehat{p}_n}

\def\hath{\widehat{h}_n}
\def\hatP{\widehat{P}_n}

\def\barp{\overline{p}_n}
\def\hatpmin{\widehat{p}_{n,\delta,-}}
\def\hatPmin{\widehat{P}_{n,\delta,-}}
\def\hatF{\widehat{F}_{n,\delta}}

\def\conf{\widehat{\sf cf}_{n,\delta}}
\def\Good{{\bf {\sf Good}}}

\def\supp{{\rm supp}}
\def\cont{{\rm ct}}

\usepackage{cite}

\ifCLASSINFOpdf
\else
\fi
\begin{document}
%
\title{Stochastic Processes with Random Contexts: a Characterization, and Adaptive Estimators for the Transition Probabilities}
%
%
%

\author{Roberto Imbuzeiro Oliveira
\thanks{R. I. Oliveira is with Instituto Nacional de Matem\'{a}tica Pura e Aplicada (IMPA) in Rio de Janeiro, 22460-320, Brazil. This work was supported by a {\em Bolsa de Produtividade em Pesquisa} from CNPq, Brazil and by the FAPESP Center for Neuromathematics, grant \# 2013/ 07699-0 from S. Paulo Research Foundation (email: rimfo@impa.br). Copyright (c) 2014 IEEE. Personal use of this material is permitted.  However, permission to use this material for any other purposes must be obtained from the IEEE by sending a request to pubs-permissions@ieee.org.}
}

\maketitle

\begin{abstract}
This paper introduces the concept of random context representations for the transition probabilities of a finite- alphabet stochastic process. Processes with these representations generalize context tree processes (a.k.a. variable length Markov chains), and are proven to coincide with processes whose transition probabilities are almost surely continuous functions of the (infinite) past. This is similar to a classical result by Kalikow about continuous transition probabilities. Existence and uniqueness of a minimal random context representation are shown, in the sense that there exists a unique representation that ``looks into the past as little as possible"~in order to determine the next symbol. Both this representation and the transition probabilities can be consistently estimated from data, and some finite sample adaptivity properties are also obtained (including an oracle inequality). In particular, the estimator achieves minimax performance, up to logarithmic factors, for the class of binary renewal processes whose arrival distributions have bounded moments of order $2+\gamma$.\end{abstract}


%
\IEEEpeerreviewmaketitle

\section{Introduction}
\IEEEPARstart{C}{onsider} a stationary stochastic process $\bX:= (X_n)_{n\in\Z}$ where each $X_n$ takes values in a finite alphabet $A$. Generally speaking, this process has infinite memory in the sense that its transition probabilities,
\[p(a|x^{-1}_{-\infty}):= \Pr{X_0=a|X^{-1}_{-\infty}=x^{-1}_{-\infty}},\]
depend on the whole infinite past
$x^{-1}_{-\infty}=(\dots,x_{-3},x_{-2},x_{-1})\in A^{\Z_-}$, as well as on the next symbol $a\in A$. This is in contrast with Markov chains, where the transition probabilities depend only on the $k$ last steps of the process, for some fixed $k$ called the {\em order} of the chain.

Markov chains are simple and concrete processes that are amenable to analysis and estimation. By contrast, processes with infinite memory are both natural and important, in theory as well as in practice. This motivates the study of families of processes that, while potentially having infinite memory, retain some characteristics of Markov chains. Two examples stand out in the literature.

\begin{enumerate}
\item {\em Random Markov chains:} for these processes, the distribution of $X_0$ given the past is determined by first sampling a random variable $K_0\in\N$ independently of $X^{-1}_{-\infty}$ and then setting \[\Pr{X_0=a|X_{-\infty}^{-1}=x^{-1}_{-\infty},K_0=k} = p_k(a|x^{-1}_{-k}),\] where $p_k$ is a pre-specified transition kernel of order $k$. In other words, the transition probabilities are {\em mixtures} of Markov transitions with different orders. Kalikow \cite{kalikow:randommarkov} introduced these processes and showed that $\bX$ is a random Markov chain if and only if its transition probabilities depend continuously on the past (cf. \remref{CFF1} below). Kalikow also noted that any continuous specification of transition probabilities corresponds to at least one process $\bX$. Such processes continue to attract much interest under the heading of ``g-measures" \cite{GalloTakahashi_Uniqueness,BramsonKalikow_NonUniqueness,JohanssonOberg_Uniqueness,Vladas_NonUniqueness}. 
\item {\em Variable length Markov chains, or context tree processes:} this is similar to the above, except that this time $K_0$ is a almost-surely finite stopping time for the filtration $\{\sigma(X^{-1}_{-k})\}_{k\geq 1}.$ That is to say, for any $k\in\N$, whether or not $K_0=k$ depends on $X^{-1}_{-k}$ but not on anything else. There is no simple characterization of context tree processes, but several interesting classes such as renewal and Markov renewal processes are of this form. A key feature of these models, which motivated Rissanen \cite{Rissanen1983} to define them, is that a VLMC model can be much more succinct than a standard order $k$ Markov model: it allows for large order parameters for ``difficult"~pasts and smaller parameters for easier pasts. This built-in sparsity property also comes with a natural tree structure that is advantageous for algorithms. Estimators such as Rissanen's original Context algorithm and BIC-penalized maximum likelihood have been introduced and analyzed  \cite{BuhlmannWyner1999,Garivier2006,CsiszarTalata2006,TalataDuncan2009,LeonardiGarivier2011,GarivierLerasle2011}.  VLMCs have been applied in Data Compression \cite{Rissanen1983,WillemsEtAl1995}, Genomics \cite{Patcher2011,BuhlmannWyner1999,Bejerano2004} and Linguistics \cite{GalvesGarcia2012}, among other fields.  
\end{enumerate}

\subsection{Our contribution}
In this paper we initiate the study of a common generalization of these two classes of processes. As we will see, this new class is mathematically natural and has good estimators. Moreover, the class of processes described by this generalization is extremely broad: in fact, there seem to be no known examples of ergodic processes outside of this class (cf. the final section of \cite{GalloPaccaut2013}). 

The new class is called {\em random context processes}. We retain the basic idea that the transition probability depends on the $K_0$ last symbols, but require that $K_0$ be a {\em randomized stopping time}. That is, we require that, almost surely, $K_0<+\infty$ and for all $k\in\N$,
\[\Pr{K_0\leq k\mid X^{-1}_{-\infty}} =  \Pr{K_0\leq k\mid X^{-1}_{-k}}.\]
A more formal definition will be given in \secref{definitions}. {\em Modulo} some technicalities, this new definition corresponds to taking mixtures of context tree representations (cf. \remref{mixtures} below), and is a natural hybrid between random Markov chains and context tree processes. Randomized stopping times have appeared before in the literature on Markov chains \cite{AldousDiaconis_SST}. 

We prove several results about this class of processes.

\paragraph{Structural results} We prove that a stationary stochastic process $\bX$ is a random context process if and only if its transition probabilities are almost surely continuous functions of the infinite past. We also show that there exists a well-defined {\em minimal random context representation}, whose $K_0$ is stochastically dominated by any other choice of randomized stopping time that is compatible with $\bX$. In this sense, the minimal representation ``looks as little as possible into the past"~in order to determine the next symbol. In order to prove these results, a crucial idea is to introduce certain minorants of the transition probabilities, which are defined in \secref{definitions} and used throughout the paper.

\paragraph{Strongly consistent estimation} we construct an estimator for the transition probabilities and for the distribution of the minimal~$K_0$. The estimator tries to mimic the minimal random context representation by estimating the minorants mentioned above. The intuition behind it is that mimicking minimality leads to succinct estimators and should translate into favorable statistical properties. In fact, we will show that our estimates are strongly consistent if $\bX$ is ergodic.

\paragraph{Finite samples: adaptivity and oracle inequalities} We also prove a finite sample adaptive property of our estimator. Considering the minimal random context representation of the process $\bX$, we show that there exists a positive $C>0$ such that, with high probability, simultaneously for all infinite pasts $x$, the total variation distance $d_{\rm TV}(p(\cdot|x),\hat{P}_n(\cdot|x))$ between real and estimated transition probabilities is bounded by \[\inf_{L\in\N}\left\{\Pr{K_0>L\mid X^{-1}_{-\infty}=x} +C\,\sqrt{\frac{|A|\ln n}{N_{n-1}(x^{-1}_{-L})}}\right\},\]where $N_{n-1}(x_1^L)$ counts the number of occurrences of $x_1^L$ in $X_1^{n-1}$ (we let the quantity in curly brackets be $+\infty$ when $N_{n-1}(x^{-1}_{-L})=0$). This follows from \thmref{oracle} below (cf. the discussion in subsection \ref{sub:bound}). This result can be viewed as an oracle inequality where the reference class of estimates consists of ``truncated"~versions of the transition probabilities; moreover, the non-standard penalty term above, which depends on the empirical quantity $N_{n-1}(x^{-1}_{-L})$, is advantageous in applications. These and other issues are discussed in Subsection \ref{sub:oracle} below.

\paragraph{Near minimax optimality for renewal processes} As an illustration of our adaptivity results, we show that our estimator for the transition probabilities is minimax optimal (up to a $\sqrt{\ln n}$ term) for the class of binary renewal processes with bounded $2+\gamma$ moments ($\gamma>0$), with a natural loss function. This is the content of \thmref{minimax} below. Binary renewal processes are a natural and popular non-parametric family of context tree sources where we expect adaptivity to different pasts to be useful, as the optimal representation of a renewal process is very far from a complete binary tree. This intuition is confirmed by \thmref{minimax}. Quite surprisingly, we could not find any results similar to this one in the literature. 

\subsection{Comparison with previous results}

Our structural results are related to the large body of work on random Markov chains, also known as $g$-measures. {\em Existence} of a process with given continuous transition probabilities is easy to show \cite{kalikow:randommarkov}, and most papers have centered on {\em uniqueness} problems: given $p$, how many compatible $\bX$ are there? See e.g.  \cite{GalloTakahashi_Uniqueness,BramsonKalikow_NonUniqueness,JohanssonOberg_Uniqueness,Vladas_NonUniqueness} for some significant results. When $p$ is allowed to be discontinuous, as in this paper, even the existence problem can be fairly delicate: see \cite{GalloGarcia2013,GalloPaccaut2013} for some recent results and open questions, including whether there exist processes that are not a.s. continuous. In any case, we do {\em not} address existence and uniqueness problems here; all our results start from a pre-existing process. 

On the estimation side, there is a vast literature on the problem of estimating context tree models\cite{BuhlmannWyner1999,Garivier2006,CsiszarTalata2006,TalataDuncan2009,LeonardiGarivier2011,GarivierLerasle2011}. The vast majority of those papers considers model section via Rissanen's original Context algorithm or the BIC criterion, both of  which rely on the Kullback Leiber distance (see \cite{LeonardiGarivier2011} for a unified analysis of the two estimators). Using the KL divergence is appealing from a coding-theoretic or MDL perspective, but seems to require fairly strong assumptions on the process, such as continuity of the transition probabilities, strong mixing properties and/or that all transition probabilities are lower bounded by a positive constant. These assumptions are violated in many important settings, including that of binary renewal processes which we consider in \secref{minimax}. Our framework circumvents the KL distance entirely and avoids such assumptions while obtaining strong results. In particular, \thmref{minimax} on near minimax estimation for renewal processes seems to be the first result of its kind.  

In terms of techniques, we employ a tool of potential independent interest: an ``economical"~form of Freedman's inequality for martingales: see \lemref{freedman} for details and \remref{economy} for comments on why the ``economy"~is crucial for our application.   

\subsection{Potential applications}
Although our focus is on basic theoretical problems, we believe that random context representations might be useful in applied settings where context trees have been successful. For concreteness, we discuss two such examples. In {\em RNA sequencing}, RNA molecules are modeled by context tree processes, and the inferred trees and probabilities are used to cluster proteins into families \cite{Bejerano2004,Patcher2011,BuhlmannWyner1999}. In {\em Linguistics}, context tree models have been used to ascertain the existence of different rhythmic patterns in European and Brazilian Portuguese \cite{GalvesGarcia2012}. 

Our approach offers two potential advantages in both settings. Firstly, our transition probability estimates come with explicit oracle-type guarantees for finite samples. Secondly, a random context representation contains more information than context trees. To explain this, we note that a context length of $k$ for a certain past means that there are no dependencies in the transition probabilities that reach more than $k$ symbols into the past. This is a somewhat brittle notion of memory, in that even a minor degree of dependence might imply a large context length. By contrast, the (estimated) distribution of the $K_0$ in a random context distribution measures the {\em strength} of these dependencies. Overall, it seems that applications of RCRs, as well as the attendant theory, deserve further study. 

\subsection{Organization}

The remainder of the paper is organized as follows. \secref{notation} introduces our notation. \secref{definitions} presents the main definitions of almost sure continuity and random context representations, together with another definition that bridges the two. In \secref{kalikov.iff} we prove the structural results on a.s. continuity and (minimal) random context representations. \secref{estimation} introduces our estimator for the transition probabilities. Strong consistency and adaptivity results are discussed in \secref{strong} and \secref{oracle}, respectively. Binary renewal processes and our nearly minimax bound for them are presented in \secref{minimax}. The appendix contains some technical results that are used in the main text.

\section{Notation}\label{sec:notation}

In this paper the {\em alphabet} is a non-empty finite set $A$. $\N=\{0,1,2,\dots,\}$ is the set of natural numbers, $\Z$ is the set of integers, and $\Z_-$, $\Z_+$ are the sets of negative and positive integers, respectively ($0$ does not belong to $\Z_+$ or $\Z_-$).

$d_{\rm TV}$ denotes the total variation metric over probability measures over the same finite or countable set $S$:
\begin{eqnarray}\label{eq:defTV}d_{\rm TV}(p,q) &:= &\frac{1}{2}\sum_{a\in S} |p(a)-q(a)|\\ \nonumber &=&\sum_{a\in S}(p(a)-q(a))_+ \\ \nonumber &=& \sum_{a\in S}(q(a)-p(a))_+.\end{eqnarray}
Note that the first equality in the display is a definition, whereas the second and third are theorems that we will use repeatedly. The set $S$ will be implicit in most cases.

$\bX:= (X_n)_{n\in\Z}\in A^{\Z}$ will always denote a stationary stochastic process with alphabet $A$. We say $\bX$ is {\em binary} if $A=\{0,1\}$.  $A^*$ is the set of all finite strings over $A$ (including the empty string $o$). The length of $w\in A^*$ is denoted by $|w|$. For $w\in A^*\backslash\{o\}$ and $1\leq i\leq |w|$, $w_i$ is the $i$th character in $i$, and if $i\leq j\leq |w|$, $w_i^j$ is the substring consisting of symbols from $w_i$ to $w_j$. We also set $w^j_i=o$ when $j<i$. 

For $w\in A^*$ and $a\in A$, $wa$ is the string obtained by concatenating $w$ and $a$; that is, $wa$ has length $|w|+1$, $(wa)_{i}=w_i$ for $1\leq i\leq |w|$ and $(wa)_{|w|+1}=a$.

We define a partial order on $u,w\in A^*$ $w\preceq u$ if $w$ is a suffix of $u$, that is $|w|\leq |u|$ and $u^{|u|}_{|u|-|w|+1}=w$. Note that $o\preceq w$ always.

$A^{\Z_-}$ is the set of {\em infinite pasts}, i.e. of sequences $x=(\dots,x_{-3},x_{-2},x_{-1})$ of elements of $A$, indexed by the negative integers. Given $x\in A^{\Z_-}$ and $i,j\in \Z_-$, we define $x_i$ and $x^{j}_{i}$ similarly to above. We will often write $x^{-1}_{-\infty}$ instead of $x$ to emphasize that $x$ is an infinite past. Finally, we use $\Prpwo{past}$
to denote the measure that $X^{-1}_{-\infty}$ induces over $A^{\Z_-}$. 

We write
\[\supp^*(\bX):=\{w\in A^*\,:\,\Pr{X^{-1}_{-|w|}=w}>0\}.\]
For $a\in A$ and $w\in \supp^*(\bX)$, we define\[ p(a|w):=\Pr{X_0=a\mid X^{-1}_{-|w|}=w}.\]
We define similarly $p(a|x^{-1}_{-\infty})=\Pr{X_0=a\mid X^{-1}_{-\infty}=x^{-1}_{-\infty}}$ for $x^{-1}_{-\infty}\in A^{\Z_-}$, in the usual measure theoretic sense of conditional probabilities. We will often use the fact that $p(a|X^{-1}_{-\infty}) = \lim_{k\to +\infty}p(a|X^{-1}_{-k})$ almost surely, which is a consequence of the Martingale Convergence Theorem. 

One final convention: $\alpha/0=+\infty$ for any positive real number $\alpha$. 

\section{Definitions}\label{sec:definitions}
In this section we compile three notions that will be fundamental in all that follows, and make some simple comments on them. 

\subsection{Almost sure continuity} Given a stationary process $\bX$ and a finite string $w\in\supp^*(\bX)$, we denote the {\em continuity rate at $w$} as follows. 
\[\cont(w):=\sup\limits_{w'\in \supp^*(\bX)\,:\,w'\succeq w}d_{\rm TV}(p(\cdot|w'),p(\cdot|w)).\]
\begin{definition}We say that $\bX$ is almost surely (a.s.)  continuous if $\cont(X^{-1}_{-k})\to 0$ as $k\to +\infty$ for almost every realization of the process. \end{definition}
\begin{remark}[Relationship to context tree processes]\label{rem:context} It may be shown that a context tree process is precisely a process such that:
$\Pr{\cont(X^{-1}_{k})=0\mbox{ for some $k\in\N$}}=1.$
In this case one can show that taking $L(x)=\inf\{k\in\N\,:\, \cont(x^{-1}_{-k})=0\}$ is a choice of stopping time achieving the context tree property: that is, 
\[\mbox{ for }\Prpwo{past}\mbox{-a.e. }x\in A^{\Z_-},\forall a\in A\,:\,p(a|x) = p(a|x^{-1}_{-L(x)}).\]\end{remark}
\begin{remark}[Relationship to continuity and mixing properties]\label{rem:CFF1} Define 
\[\cont(k):=\sup\limits_{w\in\supp^*(\bX)\,:\,|w|=k}\cont(w)\,\,(k\in\N).\]
We say $\bX$ is {\em continuous} if $\cont(k)\to 0$ as $k\to \infty$. If is known that, under certain assumptions, continuity has may have consequences for the mixing properties of the process. For instance, Comets et al. \cite{CometsFernandezFerrari2002} have shown that, if $\sum_a \inf_w p(a|w)\geq \delta$ and $\cont(k)\leq C\,k^{-(1+\alpha)}$ for constants $C,\alpha,\delta>0$, then the process $\bX$ is $\phi$-mixing with polynomial rate $\approx k^{-\alpha}$. On the other hand, there seem to be no known examples of processes that are not a.s. continuous \cite{GalloPaccaut2013}, so ``$\cont(X^{-1}_{-k})\to 0$ almost surely" (and non-uniformly) is compatible with many different kinds of mixing hypotheses.\end{remark}

\section{Random context representations}
Our next definition formalizes our discussion in the Introduction. \begin{definition}\label{def:rcr}A stationary, $A$-valued process $\bX=(X_n)_{n\in\Z}$ has a random context representation if there is a stationary process $(Y_n,K_n)_{n\in\N}$, with $(Y_n,K_n)\in A\times (\N\cup\{+\infty\})$ for all $n\in\Z$, with the following properties: {\bf (i)} $(Y_n)_{n\in\Z}$ and $(X_n)_{n\in \Z}$ have the same distribution; {\bf (ii)} $K_0<+\infty$ almost surely; and {\bf (iii)} almost surely, for all $a\in A$, $k\in\N$:
\[\Pr{Y_0=a,K_0\leq k\mid Y^{-1}_{-\infty}} =  \Pr{Y_0=a,K_0\leq k\mid Y^{-1}_{-k}}.\]

The process $(Y_n,K_n)_{n\in\Z}$ is called a {\em random context representation (RCR)} of $(X_n)_{n\in\Z}$. A process with a RCR is called a {\em random context process}.\end{definition}

\begin{remark}Notice that whether or not $\bX$ has a random context representation is determined entirely by the distribution of the process. Therefore, whenever we discuss RCRs we will assume that $Y_n=X_n$ for all $n\in\N$. We may do this even when we discuss two or more representations for the same process (i.e. the two representations may be defined in the same probability space). \end{remark}

\begin{remark}[Relationship with other classes of processes] It follows from this definition that:
\[\Pr{K=k\mid X^{-1}_{-\infty}=x^{-1}_{-\infty}}=\lambda_k(x^{-1}_{-k}),\] and \[\Pr{X_0=a\mid K=k,X^{-1}_{-\infty}=x^{-1}_{-\infty}} = p_k(a|x^{-1}_{-k})\]
almost surely, where $\lambda_k:A^k\to [0,1]$ and $p_k:A\times A^k\to [0,1]$ is an order $k$ Markov transition kernel. The fact that $K_0<+\infty$ almost surely corresponds to:
\[\sum_{k=0}^{\infty}\lambda_k(x^{-1}_{-k})=1\mbox{ for $\Prpwo{past}$-a.e. } x^{-1}_{-\infty}\in A^{\Z_-}.\]
Random Markov chains correspond to having $\lambda_k$ constant (i.e.. independent of $X^{-1}_{-k}$). Context tree processes correspond to the particular case where each $\lambda_k(x^{-1}_{-k})\in \{0,1\}$ almost surely, because this corresponds to saying that $\Ind{K_0=k}$ is entirely determined by $X^{-1}_{-k}$. Thus random context processes generalize these two previous classes.\end{remark}

\begin{remark}[Mixtures of context tree transitions]\label{rem:mixtures} Using the notation of the previous remark, we define, for $x\in A^{\Z_-}$ and $\theta\in (0,1)$:
\[L_\theta(x):=\inf\left\{k\in\N\,:\, \theta<\sum_{i=0}^k\lambda_i(x^{-1}_{-i})\right\}.\]
If the inf is not well defined, we set $L_{\theta}(x)=+\infty$. By the previous remark, $L_\theta(X^{-1}_{-\infty})$
is an a.s. finite stopping time for the filtration $\{\sigma(X^{-1}_{-k})\}_{k\in\N}$. A simple calculation shows that
\begin{IEEEeqnarray*}{rCl}\Pr{X_0=a\mid X^{-1}_{-\infty}=x^{-1}_{-\infty}} &=& \sum_{k\in\N}\lambda_k(x^{-1}_{-k})\,p_k(a|x^{-1}_{-k})\\  &=& \int_{0}^1\,p_{L_\theta(x)}(a|x^{-1}_{-L_\theta(x)})\,d\theta.\end{IEEEeqnarray*}
In other words, $p(\cdot|X^{-1}_{-\infty})$ can be written as a mixture of transition probabilities which are consistent with the stopping times $L_\theta$. This is essentially a mixture of the transition probabilities of context trees.\end{remark}

\subsection{Minorants for the transition probabilities} It turns out that a process is a.s. continuous if and only if it is a random context process (cf. \thmref{kalikov.iff}). The main bridge between the two concepts will be the following definition: for $(a,w)\in A\times \supp^*(\bX)$,
\[p_-(a|w):=\inf_{w'\in \supp^*(\bX)\,:\,w'\succeq w}p(a|w').\]
For simplicity, we extend this definition to $\tilde{w}\in A^*\backslash\supp^*(\bX)$ by setting $p_-(a|\tilde{w})=p_-(a|w)$ where $w$ is the longest suffix of $\tilde{w}$ that belongs to $\supp^*(\bX)$ (this is always well defined because $o\in\supp^*(\bX)$). We note that with this definition $p(a|w)$ is increasing with $w$.
\begin{equation}\label{eq:increasing}\forall a\in A,\,\forall w,\tilde{w}\in A^*\,:\,\tilde{w}\succeq w\Rightarrow p_-(a|\tilde{w})\geq p_-(a|w).\end{equation}
The intuition for this quantity is that, given a past $X^{-1}_{-\infty}$, $p_-(a|X^{-1}_{-k})$ is the best lower bound for $p(a|X^{-1}_{-\infty})$ that one can obtain from $X_{-k}^{-1}$. This is made precise by \propref{minorantsminimal} below.
\begin{remark}\label{rem:CFF2}The quantity
\[\gamma_k:=\inf_{|w|=k}\sum_{a\in A}p_-(a|w)\]
plays a major role in the aforementioned paper of Comets et al. \cite{CometsFernandezFerrari2002} (cf. \remref{CFF1}), who also noted (implicitly) a version of \propref{minorantsas} below.\end{remark}

\section{Structural results}\label{sec:kalikov.iff}

The goal of this section is to prove our main structural theorem about random context processes.

\begin{theorem}\label{thm:kalikov.iff}A stationary process $\bX=(X_n)_{n\in\Z}$ with finite alphabet $A$ has a random context representation if and only if it is almost surely continuous. Moreover, in that case there exists a so-called {\em minimal random context representation} $(X_n,K_n)_{n\in\Z}$ such that for all $(a,k)\in A\times \N$,
\begin{equation}\label{eq:domination2}\Pr{K_0\leq k,X_0=a\mid X^{-1}_{-\infty}} = p_-(a|X^{-1}_{-k})\mbox{ a.s},\end{equation}
and any other RCR $(X_n,\tilde{K}_n)_{n\in\N}$ will satisfy
\begin{equation}\label{eq:domination}\Pr{\tilde{K}_0\leq k\mid X^{-1}_{-\infty}}\leq \Pr{K_0\leq k\mid X^{-1}_{-\infty}}\mbox{ a.s}.\end{equation}
Finally, an RCR $(X_n,\tilde{K}_n)_{n\in\N}$ achieving almost sure equality in \eqnref{domination} must satisfy \eqnref{domination2} for all $(a,k)\in A \times \N$ (with $\tilde{K}_0$ replacing $K_0$).\end{theorem}
Before we move on to proofs, let us  explain in what sense \thmref{kalikov.iff}  implies that {\em there exists an unique minimal representation.} Recall that a real-valued random variable $U$ is stochastically dominated by $V$ if for all $t\in\R$, $\Pr{U\leq t}\geq \Pr{V\leq t}$. This is interpreted as saying that $V$ is ``larger"~than $U$, an idea that can be made precise via coupling. What \eqnref{domination} means is that, conditionally on $X^{-1}_{-\infty}$, the distribution of $K_0$ in the minimal representation is stochastically dominated by the conditional law of $\tilde{K}_0$ in any other RCR. Moreover, a.s. equality holds in \eqnref{domination} if and only if the conditional distribution of $X_0,K_0$ coincides a.s. with that of $X_0,\tilde{K}_0$. 

In the remainder of the section we prove \thmref{kalikov.iff}. Some preliminary results are collected in Subsection \ref{sub:prepare.kalikov}. We  prove that ``RCR$\Rightarrow$ A.s. continuous"~in Subsection \ref{sub:RCRAS}. The reverse implication and the construction of the minimal representation are given in Subsection \ref{sub:ASRCR}. The proof finishes with the results in Subsection \ref{sub:wrapupkalikov}. 

\subsection{Preliminary propositions}\label{sub:prepare.kalikov}

In the first proposition, we relate a.s. continuity to the minorants. 
\begin{proposition}\label{prop:minorantsas}For any process $\bX$ and any $w\in \supp^*(\bX)$,
\[\frac{1-\sum_{a\in A}p_-(a|w)}{|A|}\leq \cont(w)\leq 1 - \sum_{a\in A}p_-(a|w).\]
In particular, $\bX$ is almost surely continuous if and only if $ \lim_{k\to +\infty}\sum_{a\in A}p_-(a|X^{-1}_{-k})=1$.\end{proposition}
\begin{IEEEproof} The last assertion is a clear consequence of the preceding inequalities, which we now prove.

Given any $w'\in\supp^*(\bX)$ with $w'\succeq w$, and any $a\in A$ we have $p(a|w),p(a|w')\geq p_-(a|w)$. This implies:
\begin{eqnarray*}d_{\rm TV}(p(\cdot|w'),p(\cdot|w)) &=& \sum_{a\in A}(p(a|w')-p(a|w))_+\\ \mbox{($p(a|w)\geq p_-(a|w)$)} &\leq &  \sum_{a\in A}(p(a|w')-p_-(a|w))_+\\ \mbox{($p(a|w')\geq p_-(a|w)$)}&=&\sum_{a\in A}(p(a|w')-p_-(a|w)) \\ \mbox{($\sum_{a\in A}p(a|w')=1$)}&=&  1 -\sum_{a\in A}p_-(a|w).\end{eqnarray*}
Taking the sup over $w'\succeq w$  in the support gives $\cont(w)\leq 1 -\sum_{a\in A}p_-(a|w)$. To get the lower bound in the Proposition, note that, for the same set of $w'$,
$p(a|w')\geq p(a|w)-d_{\rm TV}(p(\cdot|w'),p(\cdot|w))\geq p(a|w)-\cont(w)$. Taking the infimum over these $w'$ gives $p_-(a|w)\geq p(a|w)-\cont(w)$, and summing over $a\in A$ gives
\[\sum_{a\in A}p_-(a|w)\geq \sum_{a\in A}(p(a|w)-\cont(w)) = 1 - |A|\,\cont(w)\] or equivalently \[\cont(w)\geq\frac{1-\sum_{a\in A}p_-(a|w)}{|A|}.\]\end{IEEEproof}

The next result shows that one may recover the transition probabilities from the minorants $p_-$ assuming a.s. continuity.

\begin{proposition}\label{prop:transitionsfromminorants}If $\bX$ is almost surely continuous, then for all $a\in A$, the following identity holds almost surely:
\[p(a|X^{-1}_{-\infty})=\lim_{k\to +\infty}p_-(a|X^{-1}_{-k}).\]\end{proposition}
\begin{IEEEproof}In this proof we assume that $X^{-1}_{-\infty}$ belongs to the set of pasts such that $p(a|X^{-1}_{-\infty})=\lim_{k\to +\infty}p(a|X^{-1}_{-k})$. This set has probability $1$ (cf. the comments in the end of \secref{notation}). 

Notice first of all that \eqnref{increasing} implies that $\lim_{k\to +\infty}p_-(a|X^{-1}_{-k})$ always exists. Secondly, for all $k\leq \ell$ and all $a\in A$, $p(a|X^{-1}_{-\ell})\geq p_-(a|X^{-1}_{-k})$, and taking the limit in $\ell\to +\infty$ and then in $k\to +\infty$ gives:
\begin{equation}\label{eq:ineqtwolims}\forall a\in A\,:\,p(a|X^{-1}_{-\infty})\geq \lim_{k\to +\infty}p_-(a|X^{-1}_{-k}).\end{equation}
Now assume $X^{-1}_{-\infty}$ is such that we have a strict inequality above for some $a\in A$. Then clearly:
\begin{eqnarray*}1 &=& \sum_{a\in A}p(a|X^{-1}_{-\infty})\\ &>&\sum_{a\in A}\lim_{k\to +\infty}p_-(a|X^{-1}_{-k})\\ &=& \lim_{k\to +\infty}\sum_{a\in A}p_-(a|X^{-1}_{-k}).\end{eqnarray*}
Therefore, assuming \eqnref{ineqtwolims} is strict for some $a$ implies that $X^{-1}_{-\infty}$ belongs to the zero measure set where $\lim_{k\to +\infty}\sum_{a\in A}p_-(a|X^{-1}_{-k})<1$ (cf. \propref{minorantsas}). It follows that equality holds in \eqnref{ineqtwolims} a.s. for all $a\in A$.\end{IEEEproof}

Finally, we make precise our intuition for $p_-(a|X^{-1}_{-k})$: that by looking $k$ symbols into the past, we can only know that $p(a|X^{-1}_{-\infty})\geq p_-(a|X^{-1}_{-k})$ 
\begin{proposition}\label{prop:minorantsminimal}For all $w\in\supp^*(\bX)$, all $a\in A$, and all random context representations $(X_n,K_n)_{n\in\Z}$ of $\bX=(X_n)_{n\in\Z}$:
\begin{equation}\label{eq:asgoal2}p_-(a|w)\geq \Pr{K_0\leq |w|,X_0=a\mid X^{-1}_{-|w|}=w}.\end{equation}\end{proposition}
\begin{IEEEproof}Take some $u\in\supp^*(\bX)$ with $u\succeq w$. Note that $p(a|u)=\Pr{X_0=a\mid X^{-1}_{-|u|}=u}\geq (I),$ where \[(I):=\Pr{X_0=a,K_0\leq |w|\mid X^{-1}_{-|u|}=u}.\]We {\em claim} that $(I)= \Pr{K_0\leq |w|,X_0=a\mid X^{-1}_{-|w|}=w}$. To see this we first rewrite \[(I)=\frac{\Ex{\Pr{X_0=a,K_0\leq |w|\mid X^{-1}_{-\infty}}\Ind{X^{-1}_{-|u|}=u}}}{\Pr{X^{-1}_{-|u|}=u}}.\] We may use the definition of RCR to replace $X_{-\infty}^{-1}$ by $X_{-|w|}^{-1}$ inside the above expectation:
\[(I)=\frac{\Ex{\Pr{X_0=a,K_0\leq |w|\mid X^{-1}_{-|w|}}\Ind{X^{-1}_{-|u|}=u}}}{\Pr{X^{-1}_{-|u|}=u}}.\]
It is well-known that measure theoretic conditional probability of an event $\Pr{E\mid X^{-1}_{-|w|}}$ equals the usual conditional probability $\Pr{E\mid X^{-1}_{-|w|}=w}$ almost surely in the cylinder set $\{X^{-1}_{-|w|}=w\}$. Since $u\succeq w$, $\{X^{-1}_{-|w|}=w\}$ contains the event $\{X^{-1}_{-|u|}=u\}$, we deduce the claimed identity  \[(I)=\Pr{X_0=a,K_0\leq |w|\mid X^{-1}_{-|w|}=w}.\] 
It follows that for all $u\succeq w$, $u\in \supp^*(\bX)$, \[p(a|u)\geq (I)=\Pr{X_0=a,K_0\leq |w|\mid X^{-1}_{-|w|}=w}.\]
Taking the infimum over $u$ finishes the proof.\end{IEEEproof}
\ignore{
The inequality in the next proposition is the key to our characterization of processes with RCRs, as well as to our proof that the minimal representation is unique. 
}\begin{remark}\label{rem:increasing}Note that the proof of \propref{transitionsfromminorants} shows that $p(a|X^{-1}_{-\infty})\geq p_-(a|X^{-1}_{-k})$ almost surely. For some later results it will be convenient to have $p(a|x)\geq p_-(a|x^{-1}_{-k})$ for all $x\in A^{\Z_-}$. To achieve this, we redefine the transition probabilities for $(x,a)\in A^{\Z_-}\times A$ as follows
\begin{IEEEeqnarray*}{rCl}p(a|x) &=& \lim_{k\to +\infty}p_-(a|x^{-1}_{-k}) \\ & & \: + \left(1 - \sum_{b\in A}\lim_{k\to +\infty}p_-(b|x^{-1}_{-k})\right)\frac{1}{|A|}\end{IEEEeqnarray*}
This is well defined and measurable, since $p_-(a|x^{-1}_{-k})$ is increasing with $k$. For the same reason, $p(a|x)\geq p_-(a|x^{-1}_{-k})$ always holds. \propref{transitionsfromminorants} implies this (re)definition of $p$ is a valid choice of regular conditional probabilities.\end{remark}

\subsection{RCR $\Rightarrow$ A.s. continuous}\label{sub:RCRAS} 
\begin{lemma}\label{lem:kalikov.if}Suppose $\bX=(X_n)_{n\in\N}$ has a random context representation. Then $(X_n)_{n\in\N}$ is a. s. continuous.\end{lemma}
\begin{IEEEproof}By \propref{minorantsas}, what we need to show is that the non-decreasing sequence
$\{\sum_{a\in A}p_-(a|X^{-1}_{-k})\}_{k\in\N}$ converges to $1$ almost surely. Since $0\leq \sum_{a\in A}p_-(a|X^{-1}_{-k})\leq 1$, there is no question that the sequence always converges, and one may deduce from Dominated Convergence that $\sum_{a\in A}p_-(a|X^{-1}_{-k})\to 1$ if and only if the next condition is satisfied:
\begin{equation}\label{eq:kalikov.if}{\bf (want)}\;\;\lim_k\Ex{\sum_{a\in A}p_-(a|X^{-1}_{-k})}=1.\end{equation}
To prove \eqnref{kalikov.if}, fix a RCR $(X_n,K_n)_{n\in\Z}$ for $\bX$ and use \propref{minorantsminimal} to obtain that, for all $k\in\N$,
\begin{eqnarray*}\sum_{a\in A}p_-(a|X^{-1}_{-k})&\geq& \sum_{a\in A}\Pr{K_0\leq k,X_0=a\mid X^{-1}_{-k}}\\ &=&\Pr{K_0\leq k\mid X^{-1}_{-k}}.\end{eqnarray*}
Taking expectations and then limits, we have that $\lim_k\Ex{\sum_{a\in A}p_-(a|X^{-1}_{-k})}$  is at least
\[\lim_k\Ex{\Pr{K_0\leq k\mid X^{-1}_{-k}}} = \lim_k\Pr{K_0\leq k},\]
which equals $1$ because $K_0<+\infty$ almost surely. This implies \eqnref{kalikov.if} and finishes the proof.\end{IEEEproof}

\subsection{A.s continuous $\Rightarrow$ RCR $+$ minimal representation}\label{sub:ASRCR}
\begin{lemma}\label{lem:kalikov.onlyif}If $(X_n)_{n\in\N}$, then it has a RCR $(X_n,K_n)_{n\in\Z}$ with the following property: for all $a\in A$ and $k\in\N$,
\begin{equation}\label{eq:formulaminimal}\Pr{X_0=a,\,K_0\leq k\mid X^{-1}_{-\infty}} = p_-(a|X^{-1}_{-k})\mbox{ a.s.}.\end{equation}\end{lemma}

\begin{IEEEproof}We may assume that the same probability space where $\bX$ is defined supports a sequence $(U_n)_{n\in\Z}$ of i.i.d. random variables that are uniform over $[0,1]$ and independent from $\bX$. We define $K_n$ as follows: define the set
\[\sS_n:=\left\{k\in\N\,:\, U_n\leq \frac{p_-(X_n|X^{n-1}_{n-k})}{p(X_n|X^{n-1}_{-\infty})}\right\},\]
and set $K_n=+\infty$ if $\sS_n=\emptyset$, or $K_n=\min \sS_n$ otherwise (this definition makes sense whenever $p(X_n|X^{n-1}_{-\infty})>0$, which happens a.s.; set $K_n=+\infty$ also when $p(X_n|X^{n-1}_{-\infty})=0$). 

We {\em claim} that $(X_n,K_n)_{n\in\Z}$ is a RCR for $\bX$. To see this, note that $K_n$ is clearly measurable for all $n\in\N$, and that the process $(X_n,K_n)_{n\in\Z}$ is stationary because $(X_n,U_n)_{n\in\Z}$ is stationary and $(X_n,K_n)=f(X_n,U_n)$, for all $n$, where  $f$ is a fixed measurable function which we do not bother to specify explicitly.  

To check that $K_0<+\infty$ a.s., recall from \propref{transitionsfromminorants} that $p_-(X_0|X^{-1}_{-k})/p(X_0|X^{-1}_{-\infty})\to 1$ a.s. as $k\to +\infty$. Therefore, if $U_0<1$ (which happens almost surely), there a.s. is some $k$ such that $U_0\leq p_-(X_0|X^{-1}_{-k})/p(X_0/X^{-1}_{-\infty})$, and the set $\sS_0$ defined above is non-empty.  

We now prove \eqnref{formulaminimal}. The main point, which follows from the definition of $K_n$, is that, in the full probability event where $p(X_0|X^{-1}_{-\infty})>0$, 
\[K_0\leq k\Leftrightarrow U_0\leq \frac{p_-(X_0|X^{-1}_{-k})}{p(X_0|X^{-1}_{-\infty})}.\]
Since $U_0$ is independent from the process $\bX$, we deduce that, almost surely,
\[\Pr{K_0\leq k\mid X^{0}_{-\infty}} =  \frac{p_-(X_0|X^{-1}_{-k})}{p(X_0|X^{-1}_{-\infty})}\]
In particular, whenever $a\in A$ is such that $p(a|X^{-1}_{-\infty})>0$, \begin{IEEEeqnarray*}{lll}&\;&\Pr{X_0=a,K_0\leq k\mid X^{-1}_{-\infty}} \\=&\;& 
\Ex{\Ind{X_0=a}\,\Pr{K_0\leq k\mid X_{-\infty}^{0}}\mid X_{-\infty}^{-1}}\\ =&\;&\Ex{\Ind{X_0=a}\,\frac{p_-(X_0|X^{-1}_{-k})}{p(X_0|X^{-1}_{-\infty})}\mid X_{-\infty}^{-1}} \\ =&  &\frac{p_-(a|X^{-1}_{-k})}{p(a|X^{-1}_{-\infty})}\,\Pr{X_0=a\mid X_{-\infty}^{-1}} \\ =& &
\frac{p_-(a|X^{-1}_{-k})}{p(a|X^{-1}_{-\infty})}\,p(a|X^{-1}_{-\infty}) \\ =& & p_-(a|X^{-1}_{-k}).\end{IEEEeqnarray*}

This proves \eqnref{formulaminimal} when $p(a|X^{-1}_{-\infty})>0$. The identity can be trivially extended to the case $p(a|X^{-1}_{-\infty})=0$, as $p_-(a|X^{-1}_{-k})=0$ a.s. for all $k\in\N$ (cf. \propref{minorantsas}) and the fact that $p_-(a|X^{-1}_{-k})$ increases with $k$). 

To sum up, we have shown that $(X_n,K_n)_{n\in\N}$ is a stationary process with $K_0<+\infty$ a.s. that satisfies \eqnref{formulaminimal}. In particular, we have expressed the conditional probability of $(X_0,K_0)=(a,k)$ given the past as a function of $X^{-1}_{-k}$, namely $p_-(a|X^{-1}_{-k})$. Standard theory implies that $\Pr{X_0=a,\,K_0\leq k\mid X^{-1}_{-\infty}}=\Pr{X_0=a,\,K_0\leq k\mid X^{-1}_{-k}}=p_-(a|X^{-1}_{-k})$ almost surely, for all $k\in\N$, and therefore $(X_n,K_n)_{n\in \Z}$ is indeed a RCR for $\bX$, with the conditional distribution of $X_0,K_0$ specified by \eqnref{domination2}.\end{IEEEproof}

\subsection{Wrapping up}\label{sub:wrapupkalikov}

We now have all the elements necessary to prove \thmref{kalikov.iff}. 

\begin{IEEEproof}[Proof of \thmref{kalikov.iff}] The equivalence between almost sure continuity and existence of an RCR is contained in Lemmas \ref{lem:kalikov.if} and \ref{lem:kalikov.onlyif}. The representation called minimal, which achieves \eqnref{domination2} in the Theorem, is presented in \lemref{kalikov.onlyif}. To prove its minimality and uniqueness, take another representation $(X_n,\tilde{K}_n)_{n\in\Z}$. \propref{minorantsminimal} implies:
\[\Pr{X_0=a,\tilde{K_n}\leq k\mid X^{-1}_{-\infty}}\leq p_-(a|X^{-1}_{-k}),\]
and summing the two sides over $a\in A$ gives \eqnref{domination}. The same reasoning implies that, for any $(a,k)\in A\times\N$:
\begin{IEEEeqnarray*}{c} \Pr{\Pr{X_0=a,\tilde{K_n}\leq k\mid X^{-1}_{-\infty}}<p_-(a|X^{-1}_{-k})}>0\\ \noalign{\noindent implies}\\ \Pr{\Pr{\tilde{K_n}\leq k\mid X^{-1}_{-\infty}}<\sum_{a\in A}p_-(a|X^{-1}_{-k})}>0.\end{IEEEeqnarray*}Therefore, any representation achieving a.s. equality in \eqnref{domination} must satisfy \eqnref{domination2} almost surely.\end{IEEEproof}

\section{Estimation: definitions and basic properties}\label{sec:estimation}

From now on we consider estimation problems for a.s. continuous processes. Throughout this section, $\bX=(X_n)_{n\in\Z}$ is the process and $(X_n,K_n)_{n\in\Z}$ is the minimal random context representation afforded by \thmref{kalikov.iff}. We will assume that the data for estimation is presented in the form of a size $n$ sample $X_1^n$. In a nutshell, we will estimate the transition probabilities indirectly, by trying to estimate the minorants $p_-$ instead of $p$, and this will also allow us to estimate the distribution of $K_0$ given $X^{-1}_{-\infty}$.

We need a few simple definitions before we present our estimator. For any string $w\in A^*$, we let $N_j(w)$ denote the number of occurrences of $w$ as a substring of $X_1^j$:
\[N_{j}(w):= \#\{|w|\leq i\leq j\,:\, X^i_{i-|w|+1}=w\}.\]
Notice that $N_j(o)=j$ for the empty string $o$. For $w\in A^*$ and $a\in A$,
we define the empirical transition probabilities as follows  \begin{equation}\label{eq:defhatp}
\setlength{\nulldelimiterspace}{0pt}
\hatp(a|w):=\left\{\begin{IEEEeqnarraybox}[\relax][c]{l?s}
\frac{N_n(wa)}{N_{n-1}(w)},&if $N_{n-1}(w)>0$;\\
\frac{1}{|A|},&otherwise.
\end{IEEEeqnarraybox}\right.
\end{equation}

\subsection{Definition of the estimator}\label{sub:definitionestimator}

Our estimation procedure consists of several steps. 

\subsubsection{Counting and confidence radii} For each pair $(w,a)\in A^*\times A$ we compute $N_{n-1}(w)$, $N_n(a)$ and $\hatp(a|w)$. We also compute {\em empirical confidence radii $\conf(w,a)$} when $N_{n-1}(w)>0$. Given a pre-specified confidence parameter $\delta>0$, we set 
\begin{equation}\label{eq:deft*ndelta}t_{*,n}(\delta):=\ln\left(\frac{|A|\,(14 + 2\log _2n)\, \left(1 + \frac{n^2}{2}\right)}{\delta}\right),\end{equation}
and define
\begin{IEEEeqnarray}{rCl}\IEEEnonumber\conf(w,a)&:=& 2\sqrt{\hatp(a|w)\,\frac{t_{*,n}(\delta)}{N_{n-1}(w)}} + \\ \label{eq:defconf} & &\:+\,(2+\sqrt{2})\,\frac{t_{*,n}(\delta)}{N_{n-1}(w)} ,\end{IEEEeqnarray}
If $N_{n-1}(w)=0$, we set $\conf(w,a)=1$.(It is not actually necessary to compute $\conf(w,a)$ in this case, but it is convenient to ignore this for the time being.)
  
\subsubsection{Preliminary estimator for $p_-$. }We define:
\begin{equation}\label{eq:hatpminus}\hatpmin(a|w) = \min\limits_{w'\in A^*,\,w'\succeq w}\left(\hatp(a|w') + \conf(w',a)\right).\end{equation}
Note that $\hatpmin(a|w)$ is increasing as a function of $w$ (with the partial order $\preceq$). 

\subsubsection{Excessive bias and truncation.} Given an infinite past $x=x^{-1}_{-\infty}\in A^{\Z_-}$, recall that that $N_{n-1}(x^{-1}_{-k})=0$ implies $\conf(x^{-1}_{-k},a)=1$. This implies the definition below makes sense.
\[\hath(x^{-1}_{-\infty}):=\min\left\{h\in\N\,:\,\sum_{a\in A}\hatpmin(a|x^{-1}_{-h})\geq 1\right\}.\]
$\hath(x^{-1}_{-\infty})$ will be our ``empirical truncation point", where the bias introduced by adding $\conf(w,a)$ in \eqnref{hatpminus} becomes excessive. Notice that, if $\hath(x)>0$, then
\[\widehat{Z}_n(x):=\sum_{b\in A}(\hatpmin(b|x^{-1}_{-\hath(x)})-\hatpmin(b|x^{-1}_{-\hath(x)+1}))>0.\]
In fact this remains true even if $\hath(x)=0$ if we adopt the convention that $\sum_{b\in A}\hatpmin(b|x^{-1}_{+1})=0$ (we will adopt this from now on). We will also use the definition 
\[\widehat{z}_n(x):=1 - \sum_{b\in A}\hatpmin(b|x^{-1}_{-\hath(x)+1}),\]
noting that $0\leq \widehat{z}_n(x)\leq \widehat{Z}_n(x)$.
\subsubsection{The actual estimates} We define our estimate for $p(a|x^{-1}_{-\infty})$ as follows. \begin{IEEEeqnarray}{rCl}\IEEEnonumber \hatP(a|x^{-1}_{-\infty})&:=& \left(1-\frac{\widehat{z}_n(x)}{\widehat{Z}_n(x)}\right)\,\hatpmin(a|x^{-1}_{-\hath(x)+1})\\ \label{defhatPoba}& &  \:+\,\left(\frac{\widehat{z}_n(x)}{\widehat{Z}_n(x)}\right)\,\hatpmin(a|x^{-1}_{-\hath(x)}).\end{IEEEeqnarray}
One may check that $\hatP(\cdot|x^{-1}_{-\infty})$ is a bona fide probability distribution over $A$ for all $x^{-1}_{-\infty}\in A^{\Z_-}$.

Our definitive estimates for the $p_-$ are as follows.
\[\hatPmin(a|x^{-1}_{-k}):=\left\{\begin{array}{ll}\hatpmin(a|x^{-1}_{-k}), & k<\hath(x^{-1}_{-\infty});\\ \hatP(a|x^{-1}_{-\infty}), & k\geq \hath(x^{-1}_{-\infty}).\end{array}\right.\]
Finally, we have our estimate for the conditional cumulative distribution function of $K_0$:
\[\hatF(k|x^{-1}_{-k}):=\sum_{a\in A}\hatPmin(a|x^{-1}_{-k}).\]

\begin{remark}[Computation of the estimator]\label{rem:compute} One way to compute this estimator is to keep an enhanced suffix tree for $X_1^n$, where each node $w$ of the tree stores the values of $N_{n-1}(w)$ and $N_n(wa)$ for $a\in A$. While in principle the estimator involves $w'$ that do not appear in the sample, we have $\conf(w',a)=1$ in those cases, and we may disregard these values in computing $\hatP(\cdot|x)$. This then a matter of walking down the tree and computing the appropriate quantities. This can be trivially implemented in $\bigoh{|A|^2n^2}$ time and requires $\bigoh{|A|n}$ space.\end{remark}

\subsection{Basic results on the estimator}

In this subsection we introduce some of the basic mathematical properties of our estimator. This will culminate with the main theorems on strong consistency (\thmref{strong}) and adaptivity (\thmref{oracle}). The first result is a fairly simple proposition.

\begin{proposition}\label{prop:simpleproposition}Given $x=x^{-1}_{-\infty}\in A^{\Z_-}$ we have the inequalities 
\[\hatpmin(a|x^{-1}_{-\hath(x)+1})\leq \hatPmin(a|x)\leq \hatpmin(a|x^{-1}_{-\hath(x)}).\]Moreover,  $\hatF(k|x^{-1}_{-k}) = \min\{1,\sum_{a\in A}\hatpmin(a|x^{-1}_{-k})\}$ for all $k\in\N$.\end{proposition}
\begin{IEEEproof}By the definition of $\hatP(a|x)$, and the fact that $0\leq \widehat{z}_n(x)/\widehat{Z}_n(x)\leq 1$, we see at once that $\hatP(a|x)$ is a convex combination of \[\hatpmin(a|x^{-1}_{-\hath(x)+1})\mbox{ and }\hatpmin(a|x^{-1}_{-\hath(x)}).\] This gives the claimed inequalities for $\hatP(a|x)$ once we recall that $\hatpmin(a|x_{-k}^{-1})$ is increasing with $k$.
In order to compute $\hatF(k|x^{-1}_{-k})$, we consider two cases.
\begin{IEEEeqnarray*}{rCl}k<\hath(x)&\Rightarrow &\sum_{a\in A}\hatpmin(a|x^{-1}_{-k})< 1, \\ && \hatPmin(a|x^{-1}_{-k}) = \hatpmin(a|x^{-1}_{-k}),\\ && \mbox{and } \hatF(k|x^{-1}_{-k}) = \sum_{a\in A}\hatpmin(a|x^{-1}_{-k});\\
k\geq \hath(x)&\Rightarrow &\sum_{a\in A}\hatpmin(a|x^{-1}_{-k})\geq 1,\\ & &  \hatPmin(a|x^{-1}_{-k}) = \hatP(a|x^{-1}_{-k}) \\ & &\mbox{ and } \hatF(k|x^{-1}_{-k}) =\sum_{a\in A}\hatP(a|x^{-1}_{-k})=1.\end{IEEEeqnarray*}\end{IEEEproof}

The next Lemmas are more technical and we defer their proofs to the Appendix.  We need two important definitions to state them. For $(w,a)\in A^*\times A$ with $N_{n-1}(w)>0$, define what one might call ``semi-empirical probabilities".
\begin{equation}\label{eq:defbarp}\barp(a|w):= \frac{\sum_{i=|w|+1}^n\Ind{X^{i-1}_{i-|w|}=w}\,p(a|X^{i-1}_{1})}{N_{n-1}(w)}.\end{equation}
For completeness, we set $\barp(a|w)=1/|A|$ when $(w,a)\in A^*\times A$ and $N_{n-1}(w)=0$. Define also the good event:
\begin{equation}\label{eq:defgood}\Good:=\bigcap\limits_{(w,a)\in A^*\times A}\left\{\begin{array}{c}|\barp(a|w)-\hatp(a|w)|\\ \leq \conf(w,a)\end{array}\right\}\end{equation}

\begin{lemma}[Proven in Subsection \ref{sub:domempirical}]\label{lem:domempirical}When $\Good$ holds, the following property holds almost surely: for all $w,w'\in A^*$ with $w'\succeq w$,
\[p_-(a|w)\leq \hatpmin(a|w)\leq \barp(a|w') + 2\conf(w',a).\]\end{lemma}

\begin{lemma}[Proven in Subsection \ref{sub:probgood}]\label{lem:probgood}$\Pr{\Good}\geq 1-\delta$.\end{lemma} 

\section{Finite sample adaptivity of the estimator}\label{sec:oracle}

In this section we prove our main result on the adaptive nature of our estimator.  Recall our convention that $\alpha/0=+\infty$ for $\alpha>0$.

\begin{theorem}[Adaptivity]\label{thm:oracle}Given a stationary process $\bX=(X_n)_{n\in\Z}$, a sample size $n\in\N$, and a confidence parameter $\delta\in (0,1)$, the estimator defined in Subsection \ref{sub:definitionestimator} satisfies the following property with probability $\geq 1-\delta$: for every $x=x^{-1}_{-\infty}\in A^{\Z_-}$ 
\begin{IEEEeqnarray}{rC}\IEEEnonumber & d_{\rm TV}(p(\cdot|x),\hatP(\cdot|x))\\ \label{eq:oracleineq} \leq &\inf_{L\in\N}\left\{1 - \sum_{a\in A}p_-(a|x^{-1}_{-L}) +c\,\sqrt{\frac{|A|\,t_{*,n}(\delta)}{N_{n-1}(x^{-1}_{-L})}}\right\}\end{IEEEeqnarray}
where $t_{*,n}(\delta)$ was defined in \eqnref{deft*ndelta} and $c=(8+2\sqrt{2})$ is universal. (Here we assume that $p$ is as defined in \remref{increasing}.)\end{theorem}

This result is interesting in its own right, and will also be used in the proof of \thmref{strong} on strong consistency. Since its statement might not be entirely transparent, we briefly discuss it before presenting the proof in Subsection \ref{proof:oracle}. 

\subsection{Some aspects of the result}

\subsubsection{On the bound in the Theorem}\label{sub:bound} The optimization over $L$ in the RHS of \eqnref{oracleineq} may be restricted to $L\in\N$ with $N_{n-1}(x^{-1}_{-L})>0$. Clearly, the bound is vacuous unless $|A|
\ll n$, and we may simplify the bound as follows
\begin{IEEEeqnarray}{rC}\label{eq:oracleineq3}\IEEEnonumber & d_{\rm TV}(p(\cdot|x),\hatP(\cdot|x))\\ \label{eq:oracleineq4}\leq & \inf\limits_{L\in\N}\left\{1 - \sum_{a\in A}p_-(a|x^{-1}_{-L})  + C\,\sqrt{\frac{|A|\,\ln(n/\delta)}{N_{n-1}(x^{-1}_{-L})}}\right\}\end{IEEEeqnarray}
with $C>0$ universal. Another point is that taking $L$ with $N_{n-1}(x^{-1}_{-L})\leq \ln n$ gives vacuous results. In particular, when the transition probabilities are lower bounded by some constant $a>0$,  $N_{n-1}(x^{-1}_{-L})$ decays exponentially with $L$, and the range of useful $L$ is restricted to $\bigoh{\ln n}$. On the other hand, there {\em are} processes where taking much larger $L$ is desirable. This is the case with the renewal processes discussed in \thmref{minimax} below. 

\subsubsection{Why \thmref{oracle} is an oracle inequality}\label{sub:oracle} Let $(X_n,K_n)_{n\in\N}$ be the minimal RCR of $\bX=(X_n)_{n\in\N}$. Take a symbol $?\not\in A$ and, for each $L\in\N$, define a conditional distribution $p_L(\cdot|x)$ on $A\cup\{?\}$ as follows:
\[p_L(a|x):=\left\{\begin{array}{ll}p_-(a|x^{-1}_{-L}),& a\in A; \\ 1 - \sum_{b\in A}p_-(b|x^{-1}_{-L}),&a=?.\end{array}\right.\]
Equivalently, $p_L(\cdot|X^{-1}_{-\infty})$ is the law of a random symbol $\widetilde{X}_0$, where $\widetilde{X}_0=X_0$ if $K_0\leq L$ and $\widetilde{X}_0=?$ if $K_0>L$. Clearly,
\begin{IEEEeqnarray*}{rCl}d_{\rm TV}(p(\cdot|x),p_L(\cdot|x)) &=& \sum_{a\in A\cup\{?\}}(p_L(a|x)-p(a|x))_+ \\ &=& 1 - \sum_{a\in A}p_-(a|x^{-1}_{-L}).\end{IEEEeqnarray*}
This definition allows us to rephrase \thmref{oracle} as an oracle inequality: with probability $\geq 1-\delta$, for all $x\in A^{\Z_-}$:
\begin{IEEEeqnarray}{rC}\IEEEnonumber & d_{\rm TV}(p(\cdot|x),\hatP(\cdot|x))\\ \label{eq:oracleineq2}\leq & \inf_{L\in\N}\left\{d_{\rm TV}(p(\cdot|x),p_L(\cdot|x)) + c\,\sqrt{\frac{|A|\,t_{*,n}(\delta)}{N_{n-1}(x^{-1}_{-L})}}\right\}.\end{IEEEeqnarray}
That is, for any given past $x$, the estimate for $p(\cdot|x)$ is as good as the best estimate of the form $p_L(\cdot|x)$ up to a penalty term that is of the order $\sqrt{\ln (n|A|/\delta)/N_{n-1}(x^{-1}_{-L})}$. Note that this is a sharp inequality: the constant in front of $d_{\rm TV}(p(\cdot|x),p_L(\cdot|x))$ in the RHS is $1$. Moreover, it is ``pastwise"~in that the $L$ achieving the infimum in the RHS of \eqnref{oracleineq2} may depend on $x$.  

\subsubsection{The sample-dependent penalty term}\label{sub:empiricalpenalty} Note that the second term in the RHS of both \eqnref{oracleineq} and \eqnref{oracleineq2} depends on an empirical quantity $N_{n-1}(x^{-1}_{-L})$. This is not standard, but we view it as a blessing in disguise. On the one hand, \eqnref{oracleineq2} holds under essentially no assumptions, and this is essential for our strong consistency result (\thmref{strong}). 

On the other hand, given appropriate mixing rates for the process, it is often possible to replace $N_{n-1}(x^{-1}_{-L})$ by $n\,\Pr{X^{-1}_{-L}=x^{-1}_{-L}}$ for all $L$ that are not ``too large". This implies that, if we are so inclined, we may replace the empirical $N_{n-1}(x^{-1}_{-L})$ term with its non-empirical counterpart over a restricted range of $L$. The key point, however, is that \thmref{oracle} accommodates this restriction {\em without any a-priori knowledge about the decay of mixing coefficients of $\bX$.} 

\subsubsection{How to apply the Theorem to context tree processes}\label{sub:minimax} Recall from \remref{context} that a process is a context tree process if and only if the continuity rates $\cont(X_{-k}^{-1})$ satisfy $\Pr{\cont(X^{-1}_{-k})=0\mbox{ for some $k\in \N$}}=1$. Combining this remark with \propref{minorantsas} shows that this holds if and only if for $\Prpwo{past}-a.e.$ $x\in A^{\Z_-}$ there exists a $L(x)\in\N$ such that $p_-(a|x^{-1}_{-L(x)})=p(a|x)$. Plugging this back into the Theorem gives that the following holds with probability $\geq 1-\delta$: for $\Prpwo{past}$-a.e. $x\in A^{\Z_-}$ with $N_{n-1}(x^{-1}_{-L(x)})>0$
\begin{equation}\label{eq:contexttreebound}d_{\rm TV}(p(\cdot|x),\hatP(\cdot|x))\leq C\,\sqrt{\frac{|A|\,\log(n|A|/\delta)}{N_{n-1}(x^{-1}_{-L(x)})}},\end{equation}
with $C>0$ universal. This is not necessarily the best way to apply \eqnref{oracleineq} for any given process: a smaller value of $L$, with greater bias, might be compensated by a much smaller ``variance" term. However, \eqnref{contexttreebound} suffices to prove nearly minimax bounds for renewal processes (\thmref{minimax}). 

\subsubsection{How to apply the Theorem to random Markov chains}\label{sub:randommarkovapply} For random Markov chains one can check that $1 - \sum_{a\in A}p_-(a|x^{-1}_{-L})\to 0$ when $L\to +\infty$, {\em uniformly in $x$}. This implies that, if $\bX$ is ergodic, the RHS of \eqnref{oracleineq} may be made uniformly small for all $x\in \supp(\bX)$, by taking a large $L$ and then letting $n\to +\infty$. We omit the details, since a similar argument is given in the first part of the proof of \thmref{strong} below.

\subsection{Proof of the adaptivity theorem} \label{proof:oracle}

\begin{IEEEproof}[Proof of \thmref{oracle}] Given $(w,a)\in A^*\times A$, consider the events
\begin{IEEEeqnarray}{rCl}\IEEEnonumber E^-_{w,a}&:=&\{p_-(a|w)\leq \hatpmin(a|w)\},\\
\IEEEnonumber E^+_{w,a}&:=&\{\hatpmin(a|w)\leq \inf_{w'\succeq w}(\barp(a|w')+2\conf(w',a))\}.\end{IEEEeqnarray}

Also set $E=\cap_{(w,a)\in A^*\times A}(E^+_{w,a}\cap E^-_{w,a})$. Combining \lemref{domempirical} and \lemref{probgood} reveals $\Pr{E}\geq \Pr{\Good}\geq 1-\delta$. The upshot is that the following deterministic statement implies the Theorem. 

\begin{center}{\em When $E$ holds, \eqnref{oracleineq} is satisfied.}\end{center}

To do this, we {\em assume from now on that $E$ holds} and proceed to prove \eqnref{oracleineq}. It suffices to show that given any infinite past $x\in A^{\Z_-}$ and any number $L\in\N$ with $N_{n-1}(x^{-1}_{-L})>0$,
\begin{IEEEeqnarray}{rC}\IEEEnonumber & \mbox{\bf (want)     }d_{\rm TV}(p(\cdot|x),\hatP(\cdot|x))\\ \label{eq:oraclewant}\leq& \; 1 - \sum_{a\in A}p_-(a|x^{-1}_{-L}) + c\,\sqrt{\frac{|A|\,t_{*,n}(\delta)}{N_{n-1}(x^{-1}_{-L})}}.\end{IEEEeqnarray}
To prove \eqnref{oraclewant} we consider the cases $L<\hath(x)$ and $L\geq \hath(x)$ separately.\\

\noindent{\em Case $1$ - $k<\hath(x)$.}  \propref{simpleproposition} gives \[\hatP(a|x)\geq \hatpmin(a|x^{-1}_{-\hath(x)+1}).\] Under the event $E$ we also have $L<\hath(x)$ and \[ \hatpmin(a|x^{-1}_{-\hath(x)+1})\geq p_-(a|x^{-1}_{-\hath(x)+1}),\] so that \[p_-(a|x^{-1}_{-\hath(x)+1})\geq p(a|x^{-1}_{-L})\] because $p(a|w)$ increases with $w$. We deduce that 
\[\hatP(a|x)\geq p_-(a|x^{-1}_{-L}),\] and the formulae for total variation distance imply
\begin{eqnarray*}\nonumber d_{\rm TV}(p(\cdot|x),\hatP(\cdot|x)) &=& \sum_{a\in A}(p(a|x) - \hatP(a|x))_+\\  &\leq & \sum_{a\in A}(p(a|x) - p_-(a|x^{-1}_{-L}))_+\\ \mbox{($p\geq p_-$)} &=& \sum_{a\in A}(p(a|x) - p_-(a|x^{-1}_{-L})) \\ (\sum_{a}p(a|x)=1)&=& 1 - \sum_{a\in A}p_-(a|x^{-1}_{-L}).\end{eqnarray*}
So \eqnref{oraclewant} holds in this case.\\

\noindent {\em Case $2$ - $L\geq \hath(x)$.} In this case we observe that, when $E$ holds, 
\[\hatpmin(a|x^{-1}_{-\hath(x)})\leq \barp(a|x^{-1}_{-L}) + 2\conf(x^{-1}_{-L},a).\]
Moreover, \propref{simpleproposition} gives
\[\hatP(a|x)\leq  \hatpmin(a|x^{-1}_{-\hath(x)}).\]
Therefore
\begin{equation}\label{eq:hatPless}\hatP(a|x)\leq  \barp(a|x^{-1}_{-L}) + 2\conf(x^{-1}_{-L},a).\end{equation}
We may combine this with the inequality $\barp(a|x^{-1}_{-L}) + 2\conf(x^{-1}_{-L},a)\geq p_-(a|x^{-1}_{-L})$, which follows from $E$. Putting all of these inequalities together shows that, for any $a\in A$ 
\begin{IEEEeqnarray*}{rC}& (\hatP(a|x) - p(a|x))_+\\ \leq & \;\barp(a|x^{-1}_{-L}) + 2\conf(x^{-1}_{-L},a)-p_-(a|x^{-1}_{-L}).\end{IEEEeqnarray*}
Adding these bounds and checking the formula for $d_{\rm TV}$ allows us to conclude that \begin{IEEEeqnarray}{rCl}\nonumber d_{\rm TV}(p(\cdot|x),\hatP(\cdot|x)) & = &\sum_{a\in A}(\barp(a|x) - p_-(a|x_{-L}^{-1}))\\ \nonumber & & \:+\,2\sum_{a\in A}\conf(x^{-1}_{-L},a)\\ \nonumber \mbox{($\sum_a\barp(a|x_{-L}^{-1})=1$)} &=& 1 - \sum_{a\in A}p_-(a|x_{-L}^{-1}) \\  \label{eq:almostoraclewant}& & \:+\,2\sum_{a\in A}\conf(x^{-1}_{-L},a).\end{IEEEeqnarray}
We will now to bound the sum $2\sum_{a\in A}\conf(x^{-1}_{-L},a)$ assuming $N_{n-1}(x^{-1}_{-L})>0$. First we recall 
\begin{IEEEeqnarray}{rCl}\IEEEnonumber\conf(x^{-1}_{-L},a)&\leq & 2\,\sqrt{\frac{\hatp(a|x^{-1}_{-L})\,t_{*,n}(\delta)}{N_{n-1}(x^{-1}_{-L})}} \\ \label{eq:confbeforesum}& & \:+\, (2+\sqrt{2})\,\frac{t_{*,n}(\delta)}{N_{n-1}(x^{-1}_{-L})}.\end{IEEEeqnarray}
We wish to add the values in the RHS of \eqnref{confbeforesum} over $a$. We will use the following bound (proven via Cauchy-Schwartz):
\begin{IEEEeqnarray*}{rCl}\IEEEnonumber \sum_{a\in A}\sqrt{\hatp(a|x^{-1}_{-L})}&\leq &\sqrt{|A|}\,\sqrt{\sum_{a\in A}\hatp(a|x^{-1}_{-L})}\\ \mbox{($\sum_a\hatp(a|x^{-1}_{-L})=1$)}&=&\sqrt{|A|}.\end{IEEEeqnarray*}
This gives a bound for the sum of $\conf(x^{-1}_{-L},a)$
\begin{IEEEeqnarray*}{rCl}2\sum_{a\in A}\conf(x^{-1}_{-L},a)&\leq &4\,\sqrt{\frac{|A|\,t_{*,n}(\delta)}{N_{n-1}(x^{-1}_{-L})}} \\ & & \:+\:(4+2\sqrt{2})\,\frac{|A|\,t_{*,n}(\delta)}{N_{n-1}(x^{-1}_{-L})}.\end{IEEEeqnarray*}
If $|A|\,t_{*,n}(\delta)\leq N_{n-1}(x^{-1}_{-L})$, we obtain
\[2\sum_{a\in A}\conf(x^{-1}_{-L},a) \leq (8+2\sqrt{2})\,\sqrt{\frac{|A|\,t_{*,n}(\delta)}{N_{n-1}(x^{-1}_{-L})}},\]
and combining this with \eqnref{almostoraclewant} gives the inequality \eqnref{oraclewant}. If on the other hand $|A|\,t_{*,n}(\delta)>N_{n-1}(x^{-1}_{-L})$, the RHS of \eqnref{oraclewant} is larger than one, whereas the LHS is upper bounded by one, so \eqnref{oraclewant} is trivially true. In either case, we have shown that \eqnref{oraclewant} holds whenever $L\geq \hath(x)$ and $N_{n-1}(x^{-1}_{-L})>0$. Since the bound is trivial when $N_{n-1}(x^{-1}_{-L})=0$, and the case $L<\hath(x)$ has been taken care of, this finishes the proof. \end{IEEEproof}

\section{Strong consistency}\label{sec:strong}

The main result of this section is that our estimator is strongly consistent under ergodicity. 

\begin{theorem}[Strong consistency]\label{thm:strong}Assume $\bX$ is a stationary, ergodic and almost surely continuous process. For each $n\in\N$, define estimators $\hatPmin$, $\hatP$ and $\hatF$ as above from the sample $X_1^n$, with a choice of $\delta=\delta_n$ satisfying $\sum_n\delta_n<+\infty$ and $\lim_n{\log(1/\delta_n)/n}\to 0$. Then there exists a set $\sT\subset A^{\Z_-}$ of probability $\Prp{past}{\sT}=1$ such that, with probability $1$ over the sample $X_1^{+\infty}$, for all $x\in \sT$ and $k\in\N$
\begin{IEEEeqnarray}{rCl}
\label{eq:strong1}\hatF(k|x)&\to& \Pr{K_0\leq k\mid X^{-1}_{-\infty}=x};\\ \label{eq:strong2}
d_{\rm TV}(p(\cdot|x),\hatP(\cdot|x))&\to &0;\end{IEEEeqnarray}
when $n\to +\infty$.\end{theorem}

\begin{remark}The convergence over $x\in\sT$ is {\em not} uniform in general. It would not be hard to show that uniform convergence of our estimator holds a.s. if $\bX$ is continuous, with $\sT:=A^{\Z_-}$.\end{remark}

\begin{IEEEproof}The proof consists of three main steps. In {\em Step $1$} we define certain events $E,F,G,H$ and prove that they have probability $1$. In {\em Step 2}, we define $\sT$ with $\Prp{past}{\sT}=1$. In {\em Step $3$} we prove \eqnref{strong1} and \eqnref{strong2} under the assumption that $E\cap F\cap G\cap H\cap I$ holds. Combining the three steps finishes the proof.\\

\noindent {\em Step 1 - almost sure events.} Let us first note that for all samples sizes $n$ we may define a corresponding event $\Good=\Good^{(n)}$ via \eqnref{defgood}. We now define:
\begin{IEEEeqnarray}{rCl}\label{eq:defeventE}E&:=&\{\Good^{(n)}\mbox{ holds for all large enough $n$}\};\\
\label{eq:defeventF}F&:=&\bigcap\limits_{(w,a)\in \supp^*(\bX)\times A}\left\{\begin{array}{rcl}\conf(w,a)&\to& 0,\\ \sqrt{\frac{\ln(n/\delta_n)}{N_{n-1}(w)}}&\to& 0\end{array}\right\};\\
\label{eq:defeventG}G&:=& \bigcap\limits_{(w,a)\in \supp^*(\bX)\times A}\{\barp(a|w)\rightarrow p(a|w)\};\\ 
\label{eq:defeventH}H&:=&\bigcap\limits_{x\in A^{\Z_-}}\left\{\begin{array}{c}\mbox{for all large enough $n$,}\\ \mbox{ \eqnref{oracleineq} holds with $\delta=\delta_n$}\end{array}\right\}.\end{IEEEeqnarray}
Let us prove that each of these events has probability $1$. 
We start by showing $\Pr{E^c}=0$. Indeed,
\[E^c = \{(\Good^{(n)})^c\mbox{ infinitely often}\}.\]
Since $\Pr{(\Good^{(n)})^c}\leq \delta_n$ (by \lemref{probgood}) and $\{\delta_n\}_{n\in\N}$ is summable (by assumption), $\Pr{E^c}=0$ follows from the first Borel-Cantelli Lemma.  
By the same token, \thmref{oracle} implies that the probability that \eqnref{oracleineq} does {\em not} hold for a given $n$ is $\leq \delta_n$, and we may equally deduce that $\Pr{H^c}=0$.

To argue that that $\Pr{F}=\Pr{G}=1$ will require the ergodicity and stationarity of $\bX$. Indeed, under these assumptions, for all $(w,a)\in\supp^*(\bX)\times A$, we have that, with probability one, 
\begin{IEEEeqnarray}{rCl}\label{eq:wasolely}\lim_{n\to +\infty}\frac{N_n(wa)}{n} &=& \Pr{X^{0}_{-|w|}=wa};\\ 
\lim_{n\to +\infty}\frac{N_{n-1}(w)}{n} &=& \Pr{X^{-1}_{-|w|}=w}>0\end{IEEEeqnarray}
and therefore, by the definition of $\hatp(a|w)$, we have that, almost surely,
\begin{equation}\label{eq:forgconditional}\lim_{n\to +\infty}\hatp(a|w)= \frac{\Pr{X^{0}_{-|w|}=wa}}{\Pr{X^{-1}_{-|w|}=w}} = p(a|w).\end{equation}
The a.s. limit \eqnref{wasolely} suffices to show $\Pr{F}=1$. To see this, one notes that for all $(w,a)\in \supp^*(\bX)\times A$,
\[\conf(w,a)\mbox{ is of the order }\sqrt{\frac{\ln(n/\delta_n)}{N_{n-1}(w)}}+\frac{\ln(n/\delta_n)}{N_{n-1}(w)}.\]
By the assumption on $\delta_n$, $\ln(n/\delta_n)/n\to 0$ whereas $N_{n-1}(w)/n$ converges a.s. to a positive limit, so $\conf(w,a)\to 0$. 

Finally, to prove $\Pr{G}=1$ we note that, under $E$ (which holds a.s.), $|\barp(a|w)-\hatp(a|w)|\leq \conf(w,a)$ for all large enough $n$. Moreover, under $F$ (which also has probability $1$), $\conf(w,a)\to 0$, whereas by \eqnref{forgconditional} we have $\hatp(a|w)\to p(a|w)$ almost surely. Therefore $\barp(a|w)\to p(a|w)$ almost surely as well.\\

\noindent {\em Step 2 - definition of $\sT$.} This is fairly straightforward. We let $\sT$ denote the set of all $x\in A^{\Z_-}$ that satisfy:
\begin{IEEEeqnarray*}{l}\forall k\in\N,\, \Pr{X_{-k}^{-1}=x^{-1}_{-k}}>0\mbox{ and }\\
\lim_{k}p(a|x^{-1}_{-k}) = \lim_{k}p_-(a|x^{-1}_{-k}) = p(a|x).\end{IEEEeqnarray*}

The comments in \secref{notation}, \propref{transitionsfromminorants} and \remref{increasing} imply $\Prp{past}{\sT}=1$.  

\noindent {\em Step 3 - \eqnref{strong1} and \eqnref{strong2} hold whenever $E\cap F\cap G\cap H$ holds.} We first show that 
\begin{equation}\label{eq:consistencypmin}\forall (w,a)\in \supp^*(\bX)\times A\,:\, \hatpmin(a|w)\to p_-(a|w).\end{equation} 
To see this, fix $(w,a)$ and some $\eps>0$. Choose $w'\in\supp^*(\bX)$, $w'\succeq w$, such that 
$p(a|w')\leq p_-(a|w)+\eps$. Because $\Good^{(n)}$ holds for all large enough $n$ (under $E$), \lemref{domempirical} implies that, with probability $1$, there exists $n_0\in\N$ such that $\forall n\geq n_0$, \[p_-(a|w)\leq \hatpmin(a|w)\leq \barp(a|w') + 2\conf(w',a).\]
Because of $G$, $\barp(a|w')\to p(a|w')\leq p_-(a|w)+\eps$. Because of $F$, $\conf(w',a)\to 0$. This implies that for all large $n$,
\[p_-(a|w)\leq\hatpmin(a|w)\leq p_-(a|w)+2\eps,\]
and \eqnref{consistencypmin} follows because $\eps>0$ is arbitrary. 

We are now ready to prove the limit in \eqnref{strong1}. For this we fix $x\in\sT$ and $k\in\N$. Note that  
\begin{eqnarray*}\hatF(k|x)&=& \min\{1,\sum_{a\in A}\hatpmin(a|x^{-1}_{-k})\}\mbox{ by Prop. \ref{prop:simpleproposition}}\\ &\to& \min\{1,\sum_{a\in A}p_-(a|x^{-1}_{-k})\} \\ &=& \sum_{a\in A}p_-(a|x^{-1}_{-k}) = \Pr{K_0\leq k\mid X^{-1}_{-\infty}=x}.\end{eqnarray*}
It remains to prove the limit in \eqnref{strong2}. For this we use the events $H$ and $F$. Fix $x\in\sT$. Given $L\in\N$, $x^{-1}_{-L}\in\supp^*(\bX)$ by the definition of $\sT$. Combining \eqnref{oracleineq} (which holds for large enough $n$ under $H$) with $F$ gives
\[\limsup_{n\to +\infty}d_{\rm TV}(p(\cdot|x),\hatP(\cdot|x))\leq 1 - \sum_{a\in A}p_-(a|x^{-1}_{-L}).\]
When $x\in\sT$, $p_-(a|x^{-1}_{-L})\to p(a|x)$ as $L\to +\infty$, and since $\sum_{a\in A}p(a|x)=1$, this implies:
\[\lim_{n\to +\infty}d_{\rm TV}(p(\cdot|x),\hatP(\cdot|x))=0.\]
Since $x\in\sT$ is arbitrary, this finishes the proof.\end{IEEEproof}

\section{Minimax results for binary renewal processes}\label{sec:minimax}

In this section we recall the definition of binary renewal processes and present our minimax results on them.  

\subsection{Definition and basic properties of binary renewal processes}

\begin{definition}\label{def:renewal}Let $\mu$ be a probability distribution on $\Z_+=\{1,2,3,4,\dots\}$ with finite first moment and define
\[\mu_\geq (k):= \sum_{n\geq k}\mu(n)\,\,(k\in\N\cup\{0\}).\] A stationary binary renewal process with arrival distribution $\mu$ is a stationary processes $\bX=(X_n)_{n\in\Z}$ with alphabet $A:=\{0,1\}$ with the following property: letting \[\{0\leq \tau_0<\tau_1<\tau_2<\dots\} = \{n\in\N\,:\, X_n=1\},\]
the random variables $\{\tau_0\}\cup \{\tau_i-\tau_{i-1}\,:\,i\in\Z_+\}$, are independent and, for all $k\in\N$
\begin{itemize}
\item $\Pr{\tau_0=k} = \mu_{\geq }(k)/\sum_{j\in\N}\mu_{\geq }(j)$;
\item $\Pr{\tau_i-\tau_{i-1}=k}=\mu(k)$, $\forall i\in\Z_+$.
\end{itemize}
Given constants $\Gamma,\gamma>0$, we let $\sR(\Gamma,\gamma)$ denote the family of all stationary binary renewal processes with arrival distribution $\mu$ satisfying:
\[\sum_{k\in\Z_+}\mu(k)\,k^{2+\gamma}\leq \Gamma.\]\end{definition} 

We omit the proof of the following simple fact. 

\begin{proposition}[Proof omitted]\label{prop:renewalbasic}A stationary binary renewal process $\bX$ with arrival distribution $\mu$ is a context tree process with $\Prpwo{past}$-a.s. finite context length:
\[L(x_{-\infty}^{-1}):= \inf\{j\in\N\,:\, x_{-j}=1\}.\]
Moreover, in the event $L(X^{-1}_{-\infty})=k$, the following equality holds almost surely: 
\[\Pr{X_0=1\mid X^{-1}_{-\infty}} = p(1|10^{k-1}) = \frac{\mu(k)}{\mu_\geq(k)}\]Finally, \[\Pr{X_{-k}^{-1}=10^{k-1}}=\frac{\mu_\geq(k)}{\sum_j\mu_\geq(j)}.\]\end{proposition}

\subsection{The near minimax bound}\label{sec:minimaxproof}

Given a stationary process $\bX$ with transition probabilities $p$, which induces a measure $\Prpwo{past}$ over $A^{\Z_-}$, and another choice $q:A^{\Z_-}\times A\to [0,1]$ of transition probabilities, we define a {\em loss function}
\[\ell_{\bX}(q):= \int_{A^{\Z_-}}\,d_{\rm TV}(q(\cdot|x),p(\cdot|x))\,d\Prp{past}{x}.\]
In words, $\ell_{\bX}(q)$ measures how close $q$ is to the true transition probabilities of $\bX$, when one averages over infinite pasts.

\begin{theorem}\label{thm:minimax}For any constants $\Gamma,0<\gamma\leq 1$, there exists a constant $C>0$ such that, for all $n\geq 3$, and all processes $\bX\in\sR(\Gamma,\gamma)$,
\[\Pr{\ell_{\bX}(\hatP)\leq C\,\sqrt{\frac{\ln n }{n}}}\geq 1 -\frac{C}{n^{1+\gamma/2}},\]
where $\hatP$ is the estimator defined in Subsection \ref{sub:definitionestimator}, run on a size-$n$ sample $X_1^n$ from $\bX$, with the choice of confidence parameter $\delta =\delta_n:= 1/n^2$.\end{theorem}

{\em Why this is a near-minimax result.} It is well known that binary processes $\bX=(X_n)_{n\in\N}$ with i.i.d. symbols are binary renewal processes with geometric arrival distribution. In particular, all i.i.d. Bernoulli processes with $\Pr{X_0=1} \in [1/3,2/3]$ are contained in $\sR(\Gamma,\gamma)$ when $\Gamma\geq \Gamma_0(\gamma)$. For large $n$, consider two i.i.d. processes $\bX^{(i)}=(X^{(i)}_n)_{n\in\Z}$, $i=1,2$, with $\Pr{X^{(1)}_0=1}=1/2$ and $\Pr{X^{(2)}_0=1}=1/2 + (1/100\sqrt{n})$. As we noted, these processes are both in $\sR(\Gamma,\gamma)$ for large enough $\Gamma$. A simple calculation shows that the Kullback Leiber distance between $(X^{(1)})_1^n$ and $(X^{(2)})_1^n$ is small, and it follows from Pinsker's Lemma that no procedure can distinguish the two processes with more than a small (constant) probability. In particular, any procedure that attempts to estimate the marginal probabilities of both $(X^{(1)})_1^n$ and $(X^{(2)})_1^n$ will obtain an error of more than $1/200\sqrt{n}$ on at least one of the two processes, with probability $\geq \alpha>0$ independent of $n$. The upshot is that any estimator $\hatP$ will satisfy
\[\sup_{\bX\in\sR(\Gamma,\gamma)}\Pr{\ell_{\bX}(\hatP)\geq \,\frac{1}{200\,\sqrt{n}}}\geq\alpha>0.\]In other words, the typical error bound in \thmref{minimax} is at most a $\bigoh{\sqrt{\ln n}}$ factor away from optimality.

\begin{IEEEproof}[Proof of \thmref{minimax}] We fix $\bX\in \sR(\Gamma,\gamma)$, with arrival distribution $\mu$, throughout the proof. We also fix $n\in\N$ and assume $\hatP$ is defined as in the statement of the Theorem. 

The basic strategy in this proof combines elements of Subsections \ref{sub:empiricalpenalty} and \ref{sub:minimax}. Namely, we will combine the fact that $\bX$ is a context tree process (cf. \propref{renewalbasic}) with the approximation $N_{n-1}(x^{-1}_{-k})\approx \Pr{X^{-1}_{-k}=x^{-1}_{-k}}\,n$. The key point is to find the right ``cutoff value"~ $k_*$ for which this is a good approximation, and show that it is indeed good for the vast majority of pasts. 

Because we will only obtain rough estimates, we adopt the following convention throughout this proof: $C,C_0,C_1,\dots>0$ are constants that are allowed to depend only on $\Gamma,\gamma$, and which may change from line to line. 

We begin with some preliminaries. By \propref{renewalbasic}, $L(x)<+\infty$  and 
$p(a|x) = p_-(a|x^{-1}_{-L(x)})$ $\Prpwo{past}$-almost-surely. Recalling that $\delta=n^{-2}$, \eqnref{contexttreebound} in Subsection \ref{sub:minimax} may be used to show that $\Pr{E}\geq 1-n^{-2}$, where:
\begin{equation}\label{eq:defEminimax}E:=\bigcap_{x\in A^{\Z_-}}\left\{ \begin{array}{c}d_{\rm TV}(p(\cdot|x),\hatP(\cdot|x))\\ \leq C\,\sqrt{\frac{\log n}{N_{n-1}(x^{-1}_{-L(x)})}}\end{array}\right\}.\end{equation}
Given $C_0,C_1>0$, define:
\begin{equation}\label{eq:defk*}k_*:=\max\left\{k\in\N\,:\,\mu_{\geq }(k)\geq \frac{C_0\,\log n}{n}\right\},\end{equation}
as well as the event
\begin{equation}\label{eq:defFminimax}F:=\left\{\forall 1\leq k\leq k_*\,:\,\frac{N_{n-1}(10^{k-1})}{n}\geq \frac{\mu_{\geq}(k)}{C_1}\right\}\end{equation}
\lemref{typical} in the Appendix proves that $\Pr{F}\geq 1-C\,n^{-{1-\gamma/2}}$, hence $\Pr{E\cap F}\geq 1-C\,n^{-1-\gamma/2}$. This means that it suffices to prove the following deterministic statement. 
\begin{center}{\em When $E\cap F$ holds, $\ell_{\bX}(\hatP)\leq C\sqrt{\ln n/n}$.}\end{center}
To prove this, we decompose and bound the loss function:
\begin{IEEEeqnarray}{Cl}\IEEEnonumber&  \ell_{\bX}(\hatP) \\ \IEEEnonumber =& \sum_{k=1}^{+\infty}\int_{\{L(x)=k\}}\,d_{\rm TV}(\hatP(\cdot|x),p(\cdot|x))\,d\Prp{past}{x}\\  \IEEEnonumber\leq \: &\sum_{k=1}^{k_*}\int_{\{L(x)=k\}}\, d_{\rm TV}(\hatP(\cdot|x),p(\cdot|x))\,d\Prp{past}{x}\\  \label{eq:dtvless1} & +\,\sum_{k=k_*+1}^{+\infty}\int_{A^{\Z_-}}\,\Ind{L(x)=k}\,d\Prp{past}{x}\\ \label{eq:useEhere}\leq\: & C\,\sum_{k=1}^{k_*}\int_{\{L(x)=k\}}\,\sqrt{\frac{\log n}{N_{n-1}(x^{-1}_{-k})}}\:d\Prp{past}{x}\\ \IEEEnonumber  & +\,\sum_{k=k_*+1}^{+\infty}\Prp{past}{L(x)=k}\\ \label{eq:useFhere} \leq \: & C\,\sum_{k=1}^{k_*}\Prp{past}{L(x)=k}\,\left(\sqrt{\frac{\log n}{n\,\mu_\geq(k)}}\right)\\  \IEEEnonumber & +\,\sum_{k=k_*+1}^{+\infty}\Prp{past}{L(x)=k}\\ \label{eq:lastlineufa} \leq & C\sqrt{\frac{\ln n}{n}}\,\sum_{k=1}^{+\infty}\sqrt{\mu_{\geq }(k)}.\end{IEEEeqnarray}
In this derivation, points \eqnref{dtvless1}, \eqnref{useEhere}, \eqnref{useFhere} and \eqnref{lastlineufa} need some explaining. First, we bounded $d_{\rm TV}\leq 1$ for $k>k_*$ to obtain \eqnref{dtvless1}. In \eqnref{useEhere}, we used the event $E$ to bound $d_{\rm TV}$ in each term. The event $F$ was used in \eqnref{useFhere} to obtain 
\[L(x)=k\Rightarrow x^{-1}_{-k}=10^{k-1}\Rightarrow N_{n-1}(x_{-k}^{-1})\geq\frac{\mu_{\geq }(k)\,n}{C_1}\mbox{ by $F$}.\]Finally, the last inequality \eqnref{lastlineufa} follows from two points. \propref{renewalbasic} and the fact that $\mu_{\geq}(0)=1$ imply $\Prp{past}{L(x)=k}\leq \mu_{\geq}(k)$, and this gives bounds on the terms with $k\leq k_*$ in the sum. For  $k>k_*$ we note that $\mu_{\geq }(k)$ is decreasing in $k$ and that $\mu_{\geq }(k_*+1)<C\sqrt{\ln n /n}$ (by the definition of $k_*$):
\[\mu_{\geq}(k)\leq \sqrt{\mu_{\geq}(k_*+1)\,\mu_{\geq}(k)}\leq \sqrt{\frac{C_0\log n}{n}\,\mu_{\geq }(k)}.\]
This justifies \eqnref{lastlineufa}, and we will be done once we bound $\sum_k\sqrt{\mu_{\geq }(k)}$ uniformly in terms of $\Gamma,\gamma$. But this is simple. Our moment assumption on $\mu$ means that, for any $k\geq 1$,
\[\mu_{\geq }(k) = \sum_{i\geq k}\mu(i)\leq k^{-(2+\gamma)}\,\sum_{i\geq 1}\mu(i)i^{2+\gamma}\leq \Gamma\,k^{-(2+\gamma)}.\] This implies that 
\[\sum_{k\in\Z_-}\sqrt{\mu_{\geq }(k)}\leq \sqrt{\Gamma}\,\sum_{k\in\Z_-}k^{-(1+\frac{\gamma}{2})}<+\infty.\]
We conclude from \eqnref{lastlineufa} that $\ell_{\bX}(\hatP)\leq C\sqrt{\ln n/n}$ whenever $E\cap F$ holds, and this finishes the proof.\end{IEEEproof}

\appendices

\section{Some technical results for the analysis of our estimator}

In this section we introduce a fundamental martingale inequality that is used in the analysis of our estimator, and then prove Lemmas \ref{lem:domempirical} and \ref{lem:probgood}. 

\subsection{The economical Freedman's inequality}\label{sub:freedman}

The following concentration lemma is closely related to Freedman's classical inequality \cite{Freedman1975}. 

\begin{lemma}[Economical Freedman's inequality]\label{lem:freedman} Let $(M_j,\sF_j)_{j=0}^n$ be a martingale such that $M_0=0$ and $|M_j-M_{j-1}|\leq 1$ for all $1\leq j\leq n$. Define the quadratic variation process 
\[V_n:= \sum_{j=1}^n\Ex{(M_j-M_{j-1})^2\mid\sF_{j-1}}.\]
Then for any $t\geq 1$:
\begin{equation}\label{eq:freedman1}\Pr{M_n\geq \sqrt{2t\,v} + \frac{2t}{3},\, V_n\leq v}\leq e^{-t}\,\Pr{V_n>0},\end{equation}
and
\begin{equation}\label{eq:freedman2}\Pr{M_n\geq 2\sqrt{V_n\,t}+t}\leq [7+\log_2 n]\,e^{-t}\Pr{V_n>0}.\end{equation}\end{lemma}

A version of the first inequality was proven in \cite{Freedman1975} without the $\Pr{V_n>0}$ term. The second inequality, with its ``empirical variance term"~$V_n$, is analogous to what Freedman used in \cite{Freedman1975} in a discussion of the Law of the Iterated Logarithm. See also \cite{Csiszar2002} and the more recent papers \cite{GarivierLeonardi2010,BelloniOliveira2011} for related inequalities used in similar contexts. In our setting, adding the $\Pr{V_n>0}$ term turns out to be crucial (cf. \remref{economy} below). 

\begin{IEEEproof} We first show how \eqnref{freedman2} follows from \eqnref{freedman1}. Fix $t\geq 1$. By our assumption that the increments of the martingale are bounded by $1$, we also have $V_n\leq n$. This means:
\[A:= \{M_n\geq 2\,\sqrt{V_n\,t} + t\} = \bigcup_{i=0}^{5 + \lceil\log_2 n\rceil}A_i,\]
where $A_0:= A\cap \{V_n\leq 2^{-6}\}$ and $A_i = A\cap \{2^{i-6}<V_n\leq 2^{i-5}\}$ for $i\geq 0$. 
We bound the probabilities of the $A_i$'s using \eqnref{freedman1}. For $i=0$, note that, if $t\geq 1$, then:
\[V_n\leq v_0:= 2^{-5}\Rightarrow 2\,\sqrt{V_n\,t} + t \geq t \geq \sqrt{2\,v_0\,t} + \frac{2t}{3},\]
so we can apply the Proposition with $v=v_0=2^{-6}$ to obtain:
\begin{IEEEeqnarray*}{rCl} \Pr{A_0} & \leq& \Pr{M_n\geq \sqrt{2\,v_0\,t} + \frac{2t}{3},V_n\leq v_0}\\ &\leq&  e^{-t}\,\Pr{V_n>0}.\end{IEEEeqnarray*}
For $i\geq 1$ we have:
\[2^{i-6}< V_n\leq v_i:= 2^{i-5}\Rightarrow 2\,\sqrt{V_n\,t} + t\geq \sqrt{2\,v_i\,t} + \frac{2t}{3},\]
so another application of \eqnref{freedman1} gives
\begin{IEEEeqnarray*}{rCl}\Pr{A_i}&\leq& \Pr{M_n\geq \sqrt{2\,v_i\,t} + \frac{2t}{3},V_n\leq v_i}\\ &\leq& e^{-t}\,\Pr{V_n>0}.\end{IEEEeqnarray*}
Summing these probabilities gives 
\[\Pr{A}\leq \sum_{i=0}^{5+\lceil \log_2 n\rceil}\Pr{A_i}\leq (6+\lceil \log_2 n\rceil)\,e^{-t}\,\Pr{V_n>0},\]
which implies \eqnref{freedman2} because $\lceil x\rceil \leq x+1$.

We now prove \eqnref{freedman1}. Define $\lambda=\sqrt{2tv}+2t/3$ and let $\Psi(x) := e^x-1-x$ ($x\in\R$). Choose some $\theta>0$ and note that 
$V_n=0\Rightarrow M_n=M_0=0$ almost surely. Defining
\[U_j:= e^{\theta\,M_j - \Psi(\theta)\,V_j},\:j=0,1,2\dots,n,\]
we see that 
\[\Pr{M_n>\lambda,V_n\leq v}\leq \Pr{U_n\geq e^{\theta\,\lambda-\Psi(\theta)\,v},\,V_n>0}.\]
We obtain
\begin{equation}\label{eq:comestosupermartingale} \Pr{M_n>\lambda,V_n\leq v} \leq  e^{-\theta\,\lambda+\Psi(\theta)\,v}\,\Ex{U_n\,\Ind{V_n>0}}\end{equation}
We {\em claim} that:
\begin{claim}\[\Ex{e^{\theta M_n -\Psi(\theta)\,V_n}\,\Ind{V_n>0}}\leq \Pr{V_n>0}.\]\end{claim}
Given this fact, optimizing the RHS of \eqnref{comestosupermartingale} over $\theta$ (following the steps of Bennett's inequality in \cite[page 24]{MassartStFlour}) finishes the proof. 

To prove the Claim, we first note that $(U_j,\sF_j)_j$ is a supermartingale; this is proven e.g.. in \cite[Page 32, Lemma 4]{ConcentrationIT}. Now write the event $\{V_n>0\}$ as a disjoint union:
\begin{IEEEeqnarray*}{rCl}\{V_n>0\} &=& \bigcup_{i=1}^nE_i, \mbox{ where}\\ E_i &:= &\{V_i>0\mbox{ and }V_j=0\mbox{ for }j<i\}.\end{IEEEeqnarray*}The $i$th element in this union is $\sF_{i-1}$-measurable. Combining this with the the supermartingale property gives:
\begin{IEEEeqnarray*}{rCl}\Ex{U_n\,\Ind{V_n>0}}& =& \sum_{i=1}^n\Ex{U_n\Ind{E_i}}\\ &=&\sum_{i=1}^n \Ex{\Ex{\left. U_n \right|\sF_{i-1}}\,\Ind{E_i}}\\ &\leq &\sum_{i=1}^n\Ex{U_{i-1}\,\Ind{E_i}}.\end{IEEEeqnarray*}
Now note that in event $E_i$ we have $V_{i-1}=0$ and therefore $M_{i-1}=0$ almost surely. This means that $U_i=1$ a.s. when $E_i$ holds.  Therefore, the expectations in the last display are at most $\Pr{E_i}$, and we obtain:
\[\Ex{U_n\,\Ind{V_n>0}}\leq \sum_{i=1}^n\Pr{E_i} = \Pr{V_n>0},\]
as claimed.\end{IEEEproof}

\subsection{Inequalities under the good event}\label{sub:domempirical}

We now prove \lemref{domempirical} above.
\begin{IEEEproof}[Proof of \lemref{domempirical}] That \[\hatpmin(a|w)\leq \barp(a|w')+2\conf(w',a)\] follows from $\hatpmin(a|w)\leq \hatp(a|w')+\conf(w',a)$ and the fact that $\hatp(a|w')\leq \barp(a|w')+\conf(w',a)$ in the event $\Good$. 

To upper bound $p_-(a|w)$, fix some $w'\in \supp^*(\bX)$ with $w\preceq w'$. In $\Good$ we have the inequality $\hatp(a|w') + \conf(w',a)\geq \barp(a|w')$. Moreover if $N_{n-1}(w')>0$, $\barp(a|w')$ is a convex combination of probabilities $p(a|X^{i-1}_{1})$ with $X^{i-1}_{i-|w|}=w$. We deduce that $\barp(a|w')\geq p_-(a|w)$ almost surely in $\Good$ whenever $w'\succeq w$ and $N_{n-1}(w')>0$. The same inequality extends to $N_{n-1}(w')=0$ because we have assumed $\conf(w',a)=1$ in this case. The upshot is that
\[\forall w'\in A^*\,:\, w'\preceq w\Rightarrow \barp(a|w')+\conf(w',a)\geq p_-(a|w),\]
and taking the infimum over all such $w'$ finishes the proof.\end{IEEEproof}

\subsection{The probability of the good event}\label{sub:probgood}

We now use the full power of \lemref{freedman} to prove \lemref{probgood}. 

\begin{IEEEproof}[Proof of \lemref{probgood}] We apply \lemref{freedman} above. Fix a pair $(w,a)\in A^*\times A$. We set $M_j(w,a):=0$ for $0\leq j\leq |w|$. For larger $j$, we set
\[M_j(w,a):= \sum_{i=|w|+1}^j(\Ind{X_i=a}-p(a|X^{i-1}_{1}))\,\Ind{X^{i-1}_{i-|w|}=w}.\]The point of this definition is that
\begin{equation}\label{eq:martingaleisthething}N_{n-1}(w)>0\Rightarrow \hatp(a|w) - \barp(a|w) = \frac{M_n(w,a)}{N_{n-1}(w)},\end{equation}
as one can check from the definitions of $\hatp$ \eqnref{defhatp} and $\barp$ \eqnref{defbarp}.

Define a filtration via  $\sF_0=\{\emptyset,\Omega\}$ and $\sF_j:=\sigma(X_1,\dots,X_j)$, $j\geq 1$. Let us check that the assumptions of \lemref{freedman} apply, i.e. that $\{(M_j(w,a),\sF_j)\}_{j=0}^n$ is a martingale starting from $0$, with increments increments between $-1$ and $1$. Firstly, note that $M_0(w,a)=0$ and $M_j(w,a) -M_{j-1}(w,a)=0$ when $j\leq |w|$. For larger $j$, the increment $M_j(w,a)-M_{j-1}(w,a)$ is equal to \[ (\Ind{X_j=a}-p(a|X^{j-1}_{1}))\,\Ind{X^{j-1}_{j-|w|}=w},\]
which is between $-1$ and $1$. Moreover, the conditional expectation of this increment with respect to $\sF_{j-1}$ is $0$, since \[\{X^{j-1}_{j-|w|}=w\}\in\sF_{j-1}\mbox{ and }\Pr{X_j=a|\sF_{j-1}} = p(a|X^{j-1}_{1}).\]

We deduce that \lemref{freedman} does apply. Looking at the increments $M_j(w,a)-M_{j-1}(w,a)$ again, we see at once that
\begin{IEEEeqnarray*}{cl} & \Ex{(M_{j}(w,a)-M_{j-1}(w,a))^2\mid\sF_{j-1}} \\ \leq\:\:&   \left\{\begin{array}{ll}0,& j\leq |w|;\\ p(a|X^{j-1}_{1})\,\Ind{X^{j-1}_{j-|w|}=w}, & j>|w|.\end{array}\right.\end{IEEEeqnarray*}
Adding these bounds gives an upper bound for the quadratic variation $V_n(w,a)$ of the martingale $M_n(w,a)$:
\begin{IEEEeqnarray*}{rCl}V_{n}(w,a)&\leq &\sum_{j=|w|+1}^np(a|X^{j-1}_{1})\,\Ind{X^{j-1}_{j-|w|}=w}\\ &=& N_{n-1}(w)\,\barp(a|w)\mbox{  (by \eqnref{defbarp}).}\end{IEEEeqnarray*}
Therefore, $V_n(w,a)>0$ implies $N_{n-1}(w)>0$. Setting
\[\barconf(w,a):=2\sqrt{\frac{\barp(a|w)\,t_{*,n}(\delta)}{N_{n-1}(w)}} + \frac{t_{*,n}(\delta)}{N_{n-1}(w)},\]
when $N_{n-1}(w)>0$, and $\barconf(w,a)=1$ otherwise, we see at once that, when $N_{n-1}(w)>0$,
\[\barconf(w,a)\,N_{n-1}(w)\geq 2\sqrt{t_{*,n}(\delta)\,V_{n}(w,a)} + t_{*,n}(\delta).\]
We deduce from \eqnref{martingaleisthething}, \lemref{freedman} (applied to $\pm M_{n}(w,a)$) and the choice of $t_{*,n}(\delta)$ in \eqnref{deft*ndelta} that
\begin{IEEEeqnarray}{rl}\nonumber & \Pr{|\barp(a|w)-\hatp(a|w)|\geq \barconf(w,a)}\\ 
\IEEEnonumber \leq\:&  \Pr{|M_n(w,a)|\geq 2\sqrt{t_{*,n}(\delta)\,V_{n}(w,a)} + t_{*,n}(\delta)}\\
\IEEEnonumber \leq\:& 2\,[7+\log_2 n]\,e^{-t_{*,n}(\delta)}\Pr{N_{n-1}(w)>0}\\ 
\label{eq:almostgoodfixedw}=\:& \frac{\delta}{|A|\,\left(\frac{n^2}{2}+1\right)}.\end{IEEEeqnarray}
Now define
\[\overline{\Good}:=\bigcap_{(w,a)\in A^*\times A}\{|\barp(a|w)-\hatp(a|w)|\leq  \barconf(w,a)\}.\]
Applying the union bound to \eqnref{almostgoodfixedw}, we obtain
\[1-\Pr{\overline{\Good}}\leq \frac{\delta}{\left(\frac{n^2}{2}+1\right)}\,\sum_{w\in A^*} \Pr{N_{n-1}(w)>0}.\]
The sum in the RHS counts the expected number of substrings $w\in A^*$ that show up in the sample. This number is upper bounded (deterministically) by $1+n^2/2$, so $\Pr{\overline{\Good}}\geq 1-\delta$. 

We will finish the proof by showing that \[\overline{\Good}\subset \Good.\] Comparing these two events, we see that it suffices to prove that, for any $(w,a)\in A^*\times A$,
$|\barp(a|w)-\hatp(a|w)|\leq \barconf(w,a)$ implies $|\barp(a|w)-\hatp(a|w)|\leq \conf(w,a)$. This is trivial when $N_{n-1}(w)=0$, so we assume $N_{n-1}(w)>0$ in what follows. 

Set $x:=\barp(a|w)$, $y:=\hatp(a|w)$ and $\alpha:=t_{*,n}(\delta)/N_{n-1}(w)$. With this notation, $\barconf(w,a)=2\sqrt{\alpha\,x}+\alpha$. Our starting point is $|\barp(a|w)-\hatp(a|w)|\leq \barconf(w,a)$, which reads 
\begin{equation}\label{eq:contabobinha}|x-y|\leq 2\sqrt{\alpha\,x} + \alpha.\end{equation}
We obtain  
\[x-2\sqrt{\alpha\,x}\leq y+\alpha\] 
and completing the square in the LHS gives 
\[(\sqrt{x} - \sqrt{\alpha})^2\leq (\sqrt{y} + \sqrt{2\alpha})^2.\]
We deduce \[\sqrt{x}\leq \sqrt{y}+(1+\sqrt{2})\,\alpha.\]
Plugging this back into \eqnref{contabobinha} gives
\[|\barp(a|w)-\hatp(a|w)|=|x-y|\leq 2\sqrt{\alpha\,y} + (2+\sqrt{2})\,\alpha,\]
and the RHS is equal to 
\[2\sqrt{\hatp(a|w)\,\frac{t_{*,n}(\delta)}{N_{n-1}(w)}} + (2+\sqrt{2})\,\frac{t_{*,n}(\delta)}{N_{n-1}(w)}=\conf(w,a).\]
In other words, we have showed the desired implication, which yields $\overline{\Good}\subset \Good$ and $\Pr{\Good}\geq \Pr{\overline{\Good}}\geq 1-\delta$, as desired.\end{IEEEproof}
\begin{remark}\label{rem:economy}Notice that the term in $\Pr{V_n>0}$ term in \lemref{freedman} is what ultimately allows us to perform the union bound in the way we do. Instead of adding a penalty term for each $w\in A^*$ (of which there are infinitely many choices), we ultimately only need to account for those $w$ that actually show up in the sample. This is a ``martingale version"~of an idea by M\'{o}rvai and Weiss \cite{MorvaiWeiss2008}, who consider countable alphabet processes with finite contexts. Similar inequalities have appeared in \cite{LeonardiGarivier2011,BelloniOliveira2011}.

A more standard alternative to this idea would be to set $\conf(w,a)\approx \sqrt{(t(w)+\log(\log n/\delta))/N_{n-1}(w)}$ for deterministic values $t(w)$ with $\sum_{w\in A^*}e^{-t(w)}\leq 1$. The usual way to achieve this is to take $t(w)$ linear in $|w|$; see e.g. \cite{Csiszar2002}. This would allow for a union bound without the $\Pr{V_n>0}$ term, but the near-minimax result in \thmref{minimax} would not follow, as the corresponding $\conf(w,a)$ would be too large for ``deep pasts".\end{remark}

\section{A typicality result for renewal processes}\label{sec:typicality}

We recall the definition of the class $\sR(\Gamma,\gamma)$ in \secref{minimax}. 
\begin{lemma}\label{lem:typical}Let $\bX=(X_n)_{n\in\Z}\in\sR(\Gamma,\gamma)$ with arrival distribution $\mu$. Then there exist a $C,C_0,C_1>0$ depending only on $\gamma,\Gamma$ such that the following holds. Let $k_*$ as in \eqnref{defk*}. 
\begin{equation}\label{eq:defk*2}k_*:=\max\left\{k\in\N\,:\,\mu_{\geq }(k)\geq \frac{C_0\log n}{n}\right\},\end{equation}
For any $n\in\N$, $n\geq 3$, the event
\begin{equation}\label{eq:defT}F:=\left\{\forall 1\leq k\leq k_*\,:\,N_{n-1}(10^{k-1}) > \,\frac{\mu_{\geq }(k)}{C_1}\right\}\end{equation} 
has probability $\Pr{F}\geq 1-C\,n^{-{1-\gamma/2}}$,\end{lemma}
\begin{IEEEproof}We use the notation introduced in \defref{renewal}, as well as the notation and conventions for constants $C,C_0,C_1,C_2\dots$ (depending only on $\Gamma,\gamma$) which we introduced in the proof of \thmref{minimax}.  
Recall from that section that
\[\{\tau_0\leq \tau_1\leq \tau_2\leq \dots \}=\{n\in\N\,:\, X_n=0\}\]
where $\{\tau_0\}\cup \{\tau_i-\tau_{i-1}\}_{i\in \Z_+}$ are independent, with 
\[\Pr{\tau_0=k}=\frac{\mu_{\geq}(k)}{\sum_{j\geq 0}\mu_{\geq j}(0)}\leq \mu_{\geq }(k)\,\,(k\in\Z_+)\]
and 
\[\Pr{\tau_i-\tau_{i-1}=k} = \mu(k)\,\,\,(k\in\Z_+).\]
Our first step in the proof will be to note that that there exist $C_3,C_4>0$ such that, for all $n\in\Z_+$:
\begin{equation}\label{eq:step1typical}\mbox{\bf (Step 1)\:}\Pr{\tau_{\lfloor n/C_3\rfloor}\leq n}\geq 1- C_4\,n^{-1-\frac{\gamma}{2}}.\end{equation}
Indeed, 
\begin{IEEEeqnarray}{rCl}\IEEEnonumber \Pr{\tau_{\lfloor n/C_3\rfloor}>n}&\leq& \Pr{\tau_{\lfloor n/C_3\rfloor}-\tau_0>n/2} \\ \label{eq:decomposestep1} & & \:+\, \Pr{\tau_0>n/2}.\end{IEEEeqnarray}
The first term in the RHS is a tail bound for an i.i.d. sum 
\[\sum_{i=1}^{\lfloor n/C_3\rfloor}(\tau_{i}-\tau_{i-1})\]
that has expectation 
\[\sum_{i=1}^{\lfloor n/C_3\rfloor}\Ex{\tau_{i}-\tau_{i-1}}\leq \left\lfloor \frac{n}{C_3}\right\rfloor \Ex{(\tau_1-\tau_0)^{\gamma}}^{1/\gamma}\leq \frac{n\,\Gamma^{1/\gamma}}{C_3} .\]
Therefore, if $C_3$ is large enough, the first term in the RHS of \eqnref{decomposestep1} is bounded by
\[\Pr{\sum_{i=1}^{\lfloor n/C_3\rfloor}(\tau_{i}-\tau_{i-1}-\Ex{\tau_i-\tau_{i-1}})>n/4},\]
and standard moment inequalities for i.i.d. sums give
\begin{equation}\label{eq:wantstep1}\Pr{\tau_{\lfloor n/C_3\rfloor}-\tau_0>n/2}\leq \frac{C_4}{n^{1+\gamma/2}}.\end{equation}For the other term in the RHS of \eqnref{decomposestep1}, we note the exact formula
\[\Pr{\tau_0>n/2}=\sum_{k>n/2}\mu_{\geq }(k)\]
and bound the terms in the sum individually: \begin{equation}\label{eq:gototail}\mu_{\geq }(k) = \sum_{j\geq k}\mu(j)\leq \frac{\sum_{j\geq k}j^{2+\gamma}\,\mu(j)}{k^{2+\gamma}} \leq \frac{\Gamma}{k^{2+\gamma}}.\end{equation}
So $\Pr{\tau_0>n/2}\leq C_4/n^{1+\gamma/2}$ as well. Combining this bound with \eqnref{wantstep1} and plugging them into \eqnref{decomposestep1} gives \eqnref{step1typical}.
Now take $k\in\Z_+$, $k\leq k_*$. Note that
\[\sum_{i=1}^{\lceil n/C_3\rceil}\Ind{\tau_i-\tau_{i-1}\geq k}\]
is a sum of i.i.d. Bernoulli terms with $\Pr{\tau_i-\tau_{i-1}\geq k} = \mu_{\geq }(k)\geq C_0\ln n/n$. The upshot is that, given any $\eta=\eta(\Gamma,\gamma)>0$, one may choose $C_0=C_0(\eta)$ appropriately so as to obtain the next inequality via a Chernoff bound for the binomial distribution \cite{ConcentrationIT}:
\[\Pr{\sum_{i=1}^{\lceil n/C_3\rceil}\Ind{\tau_i-\tau_{i-1}\geq k}<C_5\,\mu_{\geq}(k)\,{n}}<\frac{C_6}{n^{\eta}}.\]
Now \eqnref{gototail} implies $k_*\leq n^{1/(2+\gamma)}/C_7$. Hence by choosing $\eta$ large enough, we can take an union bound over $k\leq k_*$ and obtain that
\[\sum_{i=1}^{\lceil n/C_3\rceil}\Ind{\tau_i-\tau_{i-1}\geq k}\geq C_5\,\mu_{\geq}(k)\,n\mbox{ for all $k\leq k_*$}\]
with probability at least $1-Cn^{-1-\gamma/2}$.The proof finishes once we note that, when \eqnref{step1typical} and the event in the last display both hold, then $F$ must also hold: each time $\tau_i\leq n-1$ with $\tau_{i}-\tau_{i-1}\geq k$ gives a distinct occurrence of $10^{k-1}$ in $X_1^{n-1}$. This means that $\Pr{F}\geq 1- C\,n^{-1-\frac{\gamma}{2}}$, as desired.\end{IEEEproof}


\section*{Acknowledgment}
We thank Antonio Galves (USP), Alexandre Belloni (Duke) and Sandro Gallo (USP) for many useful discussions. We also thank the anonymous referees for their careful reading of our submission.

\ifCLASSOPTIONcaptionsoff
  \newpage
\fi



\bibliographystyle{IEEEtran}

\vfill
\begin{IEEEbiographynophoto}{Roberto Imbuzeiro Oliveira} was born in Rio de Janeiro, Brazil, in 1977. He received a B.Sc. degree from the {\em Pontif\'{\i}cia Universidade Cat\'{o}lica do Rio de Janeiro} (PUC-Rio) in 1999, a M.Sc. degree from {\em Instituto de Matem\'{a}tica Pura e Aplicada} (IMPA) in 2000, and a Ph.D. degree from the Courant Institute at New York University in 2004, all in mathematics. 

From 2004 to 2006 he was a Postdoctoral Researcher with the Physics of Information group at IBM TJ Watson Research Center.  In 2006 he joined the faculty of IMPA, Rio de Janeiro, Brazil, where he is currently an Associate Professor and Head for Graduate Studies. His main research interests are in probabilistic and combinatorial aspects of high dimensional phenomena, including concentration inequalities, random graph models, quantitative bounds for Markov chains and (more recently) statistical problems.

Dr. Oliveira has held a {\em Bolsa de Produtividade em Pesquisa} from CNPq, Brazil since 2007.  His undergraduate studies were supported by a {\em Bolsa por Desempenho Acad\^{e}mico} (full-tuition merit-based scholarship) from PUC-Rio. His graduate work was supported by a fellowship from CNPq, Brazil.\end{IEEEbiographynophoto}
\vfill




\end{document}